\documentclass[12pt]{amsart}
\usepackage[a4paper,margin=1in]{geometry}
\usepackage{microtype}
\usepackage[cal=euler]{mathalfa}
\usepackage{tikz,tikz-cd}
\usepackage{verbatim}
\usepackage{blkarray}
\usepackage{todonotes}
\usepackage{array}   
\newcolumntype{L}{>{$}c<{$}} 

\input{commands}
\usepackage[colorlinks,allcolors=blue]{hyperref}

\defcite\ful{fulton_intersection}
\defcite\fulb{fulton_young}
\def\coha{CoHA\xspace}
\def\cohas{CoHAs\xspace}

\def\tp#1#2{#1^{(#2)}} 
\def\bw{\f}
\def\Hilb{\cN}
\def\f{\mathbf f}

\begin{document}
\title{Tautological bases of CoHA modules}

\author{Hans Franzen}
\address{Faculty of Mathematics, Ruhr-University Bochum, 44780 Bochum, Germany}
\email{hans.franzen@rub.de}

\author{Sergey Mozgovoy} 
\address{School of Mathematics, Trinity College Dublin, Dublin 2, Ireland
\newline\indent
Hamilton Mathematics Institute, Dublin 2, Ireland}
\email{mozgovoy@maths.tcd.ie}

\begin{abstract}
Given a quiver, we consider its cohomological Hall algebra (CoHA) as well as CoHA modules built of cohomology groups of non-commutative Hilbert schemes.
We investigate cell decompositions of non-commutative Hilbert schemes and the corresponding (non-canonical) bases of CoHA modules.
We construct canonical bases of CoHA modules, which consist of products of Chern classes of tautological vector bundles over non-commutative Hilbert schemes.
This result generalizes classical results for Grassmannians and (partial) flag varieties.
\end{abstract}

\maketitle

\section{Introduction}
Cohomological Hall algebras (abbreviated as \cohas) were introduced in \cite{kontsevich_cohomological} as a mathematical incarnation of the algebras of BPS states proposed in string theory.
The definition of \cohas goes through the same lines as the definition of conventional Hall algebras \cite{ringel_hall}.
The benefit of the former is that cohomology groups used in the definition of \cohas are smaller and easier to describe explicitly than the spaces of functions on the groupoids of objects used in the definition of conventional Hall algebras.
\medskip

In particular, for the category of representations of a quiver $Q$,
the corresponding \coha \cH was described explicitly in \cite{kontsevich_cohomological} as a shuffle algebra.
More precisely, let $Q$ have the set of vertices $I$.
We can represent $\cH$ as a graded vector space
$$\cH=\bop_{\bd\in\bN^{I}}\cH_\bd,\qquad 
\cH_\bd=\bts_{i\in I}\La_{\bd_i}\iso
\bQ[x_{i,k}\col i\in I,1\le k\le\bd_i]^{\Si_\bd},$$
where $\La_{\bd_i}$ denotes the ring of symmetric polynomials in $\bd_i$ variables and $\Si_\bd=\prod_{i\in I}\Si_{\bd_i}$ is the product of symmetric groups.
The shuffle product on \cH is encoded by the Euler form of $Q$.
\medskip

The next natural question is a study of modules over the \coha \cH \cite{soibelman_remarks,franzen_chow,franzen_cohomology,franzen_semia,Young:20}.
A structure of an \cH-module can be introduced
on the cohomology of the moduli spaces of stable framed representations of $Q$, similar to the action of quantum affine algebras on (equivariant) cohomology of Nakajima quiver varieties \cite{nakajima_heisenberg,nakajima_quiverb}. 
More precisely, given vectors $\bw\in\bN^{I}$ 
and $\bd\in\bN^{I}$ (called \idef{framing} and \idef{dimension} vectors respectively),
we consider the moduli space $\Hilb_{\bd,\bw}$ 
that parametrizes pairs $(M,s)$, where $M$ is a $Q$-representation having dimension vector $\bd$ and $s\in\bigoplus_{i\in I}\Hom(\bC^{\bw_i},M_i)$ is such that its image generates $M$ as a $Q$-representation. 
This moduli space is called a \idef{non-commutative Hilbert scheme}.
For example, for the quiver having one vertex and no loops and natural numbers $d\le w$,
the moduli space $\Hilb_{d,w}$ can be identified with the Grassmannian $\Gr(w-d,w)$.
Similarly, for the quiver $1\to 2\to\dots\to n$
and vectors $\bw=(w,0,\dots,0)$, $\bd=(\bd_1,\dots,\bd_n)$ with $w\ge\bd_1\ge\dots\ge\bd_n\ge0$,
the moduli space $\Hilb_{\bd,\bw}$ can be identified with the space of (partial) flags.
\medskip 

The graded vector space 
$$\cM_\bw
=\bop_{\bd\in\bN^{I}}H^*(\Hilb_{\bd,\bw})$$ 
can be naturally equipped with the structure of a module over the \coha $\cH$
\cite{soibelman_remarks,franzen_chow}.
In \cite{franzen_semia,franzen_chow} one constructed an explicit epimorphism $\cH\onto\cM_\bw$
between \cH-modules and determined its kernel in terms of the shuffle product on $\cH$.
This gives, in principle, a description of cohomologies $H^*(\Hilb_{\bd,\bw})$ of non-commutative Hilbert schemes.
\medskip

The epimorphism $\cH\onto\cM_\bw$
is given for the degree ~$\bd$ components by
$$\cH_\bd\to H^*(\cN_{\bd,\bw}),\qquad
\prod_{\ov{i\in I}{1\le k\le\bd_i}}e_{i,k}^{n_{i,k}}\mto
\prod_{\ov{i\in I}{1\le k\le\bd_i}}c_k(\cU_i)^{n_{i,k}},$$
where $e_{i,k}\in\La_{\bd_i}$ is the $k$-th elementary symmetric polynomial and $c_k(\cU_i)$ is the $k$-th Chern class of the tautological vector bundle $\cU_i$ over $\Hilb_{\bd,\bw}$ at the vertex $i\in I$ .
This implies that for appropriate choices of the powers $(n_{i,k})$ we can obtain a basis of $H^*(\Hilb_{\bd,\bw})$.
\medskip

On the other hand, it was proved in  \cite{reineke_cohomology,engel_smooth}
that non-commutative Hilbert schemes $\cN_{\bd,\f}$ possess a cell decomposition, parametrized by subtrees of the tree of paths $\cP^\f$ in the framed quiver $Q^\f$.
While the set of parameters (subtrees of the tree of paths) is independent of any choices, the corresponding cells (and the classes of their closures) generally depend on a (non-canonical) choice of a total order on the tree of paths.
Taking duals of the cycles of cell closures, we obtain a (non-canonical) basis of $H^*(\Hilb_{\bd,\bw})$.
\medskip 

In this paper we relate the above two approaches to cohomology of non-commutative Hilbert schemes and,
as a result, 
we describe a canonical basis of $H^*(\Hilb_{\bd,\bw})$ in terms of the Chern classes of tautological vector bundles over $\Hilb_{\bd,\bw}$.
There exists a bijection between the set of subtrees of $\cP^\f$ and the set of certain multi-partitions \cite[\S8]{engel_smooth},
which is obtained by applying a modification of a depth-first search algorithm for plane trees.
This bijection 
plays a fundamental role in our approach, as we parametrize products of Chern classes of tautological bundles by multi-partitions and relate them to the classes of cell closures parametrized by subtrees of $\cP^\f$.
Our main result is (see Theorem \ref{th:basis2})


\begin{theorem}
Let $\cS^\f(\bd)$ denote the set of multi-partitions 
\begin{equation*}
\la=(\tp\la i)_{i\in I},\qquad
\tp\la i
=(\tp\la i_1\ge\dots\ge\tp\la i_{\bd_i}\ge0),
\qquad i\in I,
\end{equation*}
satisfying condition \eqref{Phi}.
Then the classes
$$\prod_{i\in I}\prod_{k\ge1} c_k(\cU_i)^{\tp\la i_k-\tp\la i_{k+1}}\in H^{2\n\la}(\cN_{\bd,\f}),\qquad
\la\in\cS^\f(\bd),$$
form a basis of $H^*(\cN_{\bd,\f})$,
where $\n\la=\sum_{i\in I}\n{\tp\la i}$.
\end{theorem}

\section{Cell decomposition of non-commutative Hilbert schemes}

\subsection{Non-commutative Hilbert schemes}
Let $Q$ be a quiver with the set of vertices $I=Q_0$ and the set of arrows $Q_1$.
The bilinear form $\hi=\chi_Q: \Z^I \times \Z^I \to \Z$ defined by
\begin{equation}
\chi(\bd,\be) = \sum_{i \in I} 
\bd_i\be_i - \sum_{(a: i \to j) \in Q_1} \bd_i\be_j
\end{equation}
is called the \idef{Euler form} of $Q$.
Given a vector $\bw\in\bN^{I}$, called a \idef{framing vector},
we define a new (framed) quiver $Q^\bw$ by adding a new vertex $\infty$ and $\bw_i$ arrows $\infty\to i$, for all $i\in I$.
We define a \idef{framed representation} ~$M$ to be a representation of $Q^\bw$ such that $\dim M_\infty=1$.
We call it \idef{stable} 
if $M$ is generated, as a $Q^\bw$-representation, by the subspace $M_\infty$.

For any dimension vector $\bd\in\bN^{I}$,
we define
\begin{align}
R(Q,\bd) &=\bop_{(a:i\to j)\in Q_1}\Hom(\bC^{\bd_i},\bC^{\bd_j}), \\
R^\f(Q,\bd) &= R(Q,\bd) \oplus \bop_{i\in I}\Hom(\bC^{\bw_i},\bC^{\bd_i}).
\end{align}
An element $M\in R^\f(Q,\bd)$ can be interpreted as a framed representation.
Let $R^{\f,\st}(Q,\bd)\sbs R^\f(Q,\bd)$ be the open subspace of stable framed representations.
The space $R^\f(Q,\bd)$ is equipped with the action of the group $G_\bd=\prod_{i\in I}\GL_{\bd_i}(\bC)$.
The induced action of $G_\bd$ on
$R^{\f,\st}(Q,\bd)$ is free and there exists a smooth algebraic variety
\begin{equation}
\cN_\bd
=\cN_{\bd,\bw}
=R^{\f,\st}(Q,\bd)/G_\bd
\end{equation}
parameterizing stable framed representations up to isomorphism,
called the \idef{non-commutative Hilbert scheme}.
If non-empty, it has dimension 
\begin{equation}\label{dim Nd}
\dim\cN_\bd=\bw\cdot\bd-\hi(\bd,\bd).
\end{equation}

Every representation $M\in \cN_\bd$ is equipped with a nonzero element $m_*\in M_\infty$ that generates ~$M$ as a representation of $Q^\bw$.
%
%
Let $A=\bC Q^\bw$ be the path algebra of the quiver $Q^\bw$ and let $P=Ae_\infty$ be the projective $A$-module
corresponding to the idempotent
$e_\infty\in A$ (the trivial path at~ $\infty$).
For any $M\in\cN_\bd$, we have an epimorphism
$P\to M$, $u\mto um_*$.
The moduli space $\cN_\bd$ parametrizes such epimorphisms up to an automorphism of $M$.

The projective module $P=Ae_\infty$ has a basis consisting of paths in $Q^\bw$ that start at $\infty$.
This set of paths has a tree structure which will be important in our analysis of stratifications of the moduli space $\cN_\bd$.
Therefore we will discuss these combinatorial structures in more detail in the next sections.

\subsection{Posets and trees}
Let $(\cP,\preceq)$ be a partially ordered set (poset).
We will write $u\prec v$ if $u\preceq v$ and $u\ne v$, for $u,v\in\cP$.
For any $v\in \cP$, we define
\begin{equation}
\cP_{\preceq v}=\sets{u\in \cP}{u\preceq v}\qquad
\cP_{\prec v}=\sets{u\in \cP}{u\prec v}.
\end{equation}
The elements of $\cP_{\prec v}$ are called \idef{predecessors} of $v$.
A subset $S\sbs \cP$ is called a \idef{lower set} if $\cP_{\preceq v}\sbs S$, for all $v\in S$.

We define a \idef{(rooted) tree} to be a poset $(\cP,\preceq)$ such that
\begin{enumerate}
\item There is a unique minimal element $*\in \cP$, called \idef{the root}.
\item For any $v\in \cP\ms\set*$, 
the poset $\cP_{\prec v}$ is a finite chain (totally ordered set).
Its unique maximal element is denoted by $p(v)$, called the \idef{parent} of $v$.
\end{enumerate}

We define a \idef{subtree} of a tree $\cP$ to be a non-empty lower set $S\sbs\cP$.
In what follows we will consider only finite subtrees, unless otherwise stated.
For any subtree $S\sbs\cP$,
we define its \idef{critical set} to be
\begin{equation}
C(S)
=\min\nolimits_{\preceq}(\cP\ms S)
=\sets{v\in\cP\ms S}{u\prec v\imp u\in S}.
\end{equation}
Note that if $v\in C(S)$, then $S\cup\set v$ is again a tree.

\subsection{Path poset}
For any path $u$ in $Q^\bw$, we denote by $s(u)$ and $t(u)$ its source and target vertices, respectively.
Let $\cP=\cP^\bw$ be the set of paths in $Q^\bw$ that start at $\infty$.
It is a basis of the module $P=Ae_\infty$ considered earlier.
We define the partial order on ~$\cP$ given by
\begin{equation}
u\preceq v\quad\text{ if }\quad v=wu
\end{equation}
for some path $w$ with $s(w)=t(u)$.
The poset \cP is a tree with the root $*=e_\infty$.
For any non-trivial path $u=a_n\dots a_1\in\cP$ (where $a_i$ are arrows with $s(a_i)=t(a_{i-1})$, for $1<i\le n$), the parent of $u$ is $a_{n-1}\dots a_1\in\cP$.

For any (finite) subset $S\sbs\cP$, we define its dimension vector (note that we omit the vertex $\infty$ here) 
\begin{equation}
\udim S=(\# S_i)_{i\in I},\qquad
S_i=\sets{u\in S}{t(u)=i}.
\end{equation}
For any $\bd\in\bN^I$, let $\cP(\bd)$ denote the set of trees $S\sbs \cP$ with $\udim S=\bd$.
This set is finite.
Indeed, for any $S\in\cP(\bd)$ and $v\in S$, the length of the path $v$ is $\le \# S-1=\n\bd=\sum_{i\in I}\bd_i$ and the number of paths in $\cP$ having a given length is finite.

\subsection{Properties of stable framed representations}
Recall that every $M\in\cN_\bd$ is equipped with a vector $m_*\in M_\infty$ that generates $M$ as a $Q^\bw$-representation.
For any path $u\in\cP$, we define 
\begin{equation}
m_u=um_*\in M_{t(u)}.
\end{equation}
These vectors generate $M$ as a vector space.
For any subset $S\sbs \cP$, we define the subspace
\begin{equation}
M_S=\angs{m_u}{u\in S}\sbs M.
\end{equation}
Note that if $(m_u)_{u\in S}$ is a basis of $M$, then
$\udim S=\bd$.

\begin{lemma}\label{tree generator}
Let $S\sbs\cP$ be a tree such that $m_v\in M_S$ for all $v\in C(S)$. Then $M_S=M$.
\end{lemma}
\begin{proof}
It is enough to show that $M_S\sbs M$ is a subrepresentation.
Let $u\in S$ and $a$ be an arrow with $s(a)=t(u)$.
Then $au\in S$ or $au\in C(S)$.
If $au\in S$, then $am_u=m_{au}\in M_S$ by the definition of $M_S$.
If $au\in C(S)$, then $am_u=m_{au}\in M_S$ by assumption.
\end{proof}

\begin{corollary}
\label{basis ext}
Let $S\sbs\cP$ be a tree, such that $(m_u)_{u\in S}$ are linearly independent.
Then there exists a tree $S\sbs S'\sbs\cP$ 
such that $(m_u)_{u\in S'}$ is a basis of $M$.
\end{corollary}
\begin{proof}
If $M_S=M$, then we are done.
Otherwise, by the previous lemma, there exists $v\in C(S)$ such that $m_v\notin M_S$.
Then $S'=S\sqcup\set v$ is a tree and $(m_u)_{u\in S'}$ are linearly independent. 
Then we proceed by induction.
\end{proof}

\subsection{The cells}
To define a cell decomposition of the non-commutative Hilbert scheme $\cN_\bd$, we need to fix a total order on the set of paths $\cP$.
We will say that a total order $\le$ on $\cP$ is \idef{admissible} if it extends the partial order~$\preceq$
on \cP.
We will say that a total order $\le$ on \cP
is \idef{monomial} if
\begin{enumerate}
\item $\le$ extends the partial order $\preceq$ on $\cP$.
\item If $u<v$, then $au<av$ for any compatible arrow $a$.
\item $\le$ is a well-order.
\end{enumerate}

\begin{example}[Shortlex order]
\label{shortlex}
Assume that the set of arrows ~$Q^\bw_1$
is totally ordered.
Then, for  two paths in $\cP$
\begin{equation}\label{u-v}
u=a_m\dots a_1,\qquad
v=b_n\dots b_1,
\end{equation}
we say that $u<v$ if $m<n$ or if $m=n$ and $a_k<b_k$ for the minimal $k$ with $a_k\ne b_k$.
This is a monomial order.
\end{example}

\begin{example}[Weighted shortlex order]
Let $Q^\f_1$ be totally ordered and let $\wt:Q^\f_1\to\bR_{>0}$ be a map.
For any path $u=a_m\dots a_1$, we define its weight $\wt(u)=\sum_i\wt(a_i)$.
Given two paths $u,v$ \eqref{u-v}, we say that $u<v$ if $\wt(u)<\wt(v)$ or if $\wt(u)=\wt(v)$ and $u<v$ \wrt the shortlex order.
This is a monomial order.
\end{example}

\begin{example}[Lex order]
As before, let $Q^\bw_1$ be totally ordered.
Given two paths $u,v$ \eqref{u-v}, we say that $u<v$ 
if $a_k<b_k$ for the minimal $k$ with $a_k\ne b_k$ (if such $k$ exists) 
or if $m<n$ and $a_k=b_k$ for all $1\le k\le m$.
This is the order used in 
\cite{reineke_cohomology,engel_smooth}.
It is admissible, but not monomial.
For example, for the quiver $Q^\f$ with arrows $f:\infty\to0$ and $a,b:0\to 0$ and the order on arrows $f<a<b$,
we have $af<a^2f$, but $baf>ba^2f$.
\end{example}

In what follows, we fix an admissible order on $\cP$.
For any tree $S\in\cP(\bd)$, we define the sets
\begin{enumerate}
\item 
$U_S\sbs \cN_\bd$ that consists of  $M\in\cN_\bd$ such that the tuple
$(m_u)_{u\in S}$ is a basis of $M$.
\item
$Z_S\sbs U_S$ that consists of $M\in U_S$ such that,
for every $v\in C(S)$, we have 
$m_v\in M_{S_{<v}}$.
\end{enumerate}

\begin{remark}
Note that $U_S$ is independent of the choice of the total order on $\cP$, while $Z_S$ depends on it.
It follows from Corollary \ref{basis ext} that $\cN_\bd$ is covered by the sets $U_S$.
\end{remark}

The following result was proved in \cite[Lemma 3.4]{reineke_cohomology} for the $m$-loop quiver. The proof for a general quiver is analogous. We give it for completeness.

\begin{lemma}\label{lm:dim U_S}
For any tree $S\in\cP(\bd)$, 
the set $U_S\sbs \cN_\bd$ is open.
It is an affine space of dimension 
$\#\bigcup_{i\in I}\rbr{S_i\xx C(S)_i}$.
\end{lemma}
\begin{proof}
The condition $M\in U_S$ means that the vectors $(m_u)_{u\in S}$ are linearly independent and this is an open condition. We may express $m_v$ in a unique way as a linear combination $m_v = \sum_{u \in S_i} c_{u,v} m_u$ for every $v \in C(S)_i$ and every $i \in I$. The scalars $c_{u,v}$ provide a morphism $U_S \to \mathbb{A}^N$, where 
$N=\#\bigcup_{i\in I}S_i\xx C(S)_i$.

Conversely, let us define a map $\mathbb{A}^N \to U_S$. Fix a basis $(m_u)_{u \in S_i}$ of each vector space $M_i$ for $i \in Q^{\bw}_0$. To a tuple $(c_{u,v})$ of scalars with $(u,v) \in \bigcup_{i \in I} S_i \times C(S)_i$, we assign a representation $M \in U_S$ as follows.
To specify a representation, we need to specify the maps $a:M_i\to M_j$ for all arrows $a:i\to j$ in $Q^\bw$.
This means that, for every $u\in S_i$, we need to specify $am_u\in M_j$.
If $au\in S_j$, then we set $am_u=m_{au}$.
If $au\in C(S)_j$, then we define $am_u=\sum_{u'\in S_j}c_{u',au}m_{u'}$. The thus obtained representation $M$ lies in $U_S$.
The two maps are mutually inverse.
\end{proof}

\begin{lemma}\label{lm:dim Z_S}
For any tree $S\in\cP(\bd)$, 
the set $Z_S\sbs \cN_\bd$ is locally-closed.
It is an affine space of dimension 
$\#\bigcup_{i\in I}\sets{(u,v)\in S_i\xx C(S)_i}{u<v}$.
\end{lemma}
\begin{proof}
For any $M\in U_S$, the condition $m_v\in M_{S_{<v}}$ for all $v\in C(S)_i$ is a closed condition: when writing $m_v = \sum_{u \in S_i} c_{u,v}m_u$ as in the proof of Lemma \ref{lm:dim U_S}, it is given by the vanishing of 
$c_{u,v}$ for all $u\in S_i$ with $u>v$.
This also shows that $Z_S$ is a linear subspace of $U_S$ of the required dimension.
\end{proof}

Lemma \ref{lm:dim Z_S} implies that $Z_S\sbs\cN_\bd$ has dimension
\begin{equation}\label{dim Z_S}
d(S)=\sum_{v\in C(S)}k_v,\qquad
k_v=\#\sets{u\in S_{t(v)}}{u<v}.
\end{equation}
Lemma \ref{lm:dim U_S} implies that 
\begin{equation}\label{dim U_S}
\dim\cN_\bd
=\dim U_S=\udim C(S)\cdot\bd.
\end{equation}
On the other hand, we have  $\dim\cN_\bd=\bw\cdot\bd-\hi(\bd,\bd)$
\eqref{dim Nd}.
We can prove more precisely

\begin{lemma}
\label{lm:dim CS}
For any tree $S\in\cP(\bd)$, the dimension vector of the critical set $C(S)$ is
\begin{equation}\label{dim CS}
\udim C(S)=\bw-\hi(\bd,-),
\end{equation}
meaning that $\udim C(S)\cdot\be=\bw\cdot\be-\hi(\bd,\be)$, for all $\be\in\bZ^I$.
\end{lemma}
\begin{proof}
The statement is clear for $S=\set{*}$.
Assume that the statement is true for a tree ~$S$ and let $v\in C(S)_i$ and $S'=S\cup \set v$.
Then $C(S')=C(S)\ms\set v\cup \ch(v)$, where $\ch(v)=\sets{u\in\cP}{p(u)=v}$ is the set of children of $v$ (they correspond to arrows that start at $i$).
Therefore
$$\udim C(S')=\udim C(S)-e_i+\udim\ch(v)=\udim C(S)-\hi(e_i,-).$$
If $\bd=\udim S$, then $\bd'=\udim S'=\bd+e_i$, hence
$$\udim C(S')=\bw-\hi(\bd,-)-\hi(e_i,-)=\bw-\hi(\bd',-)$$ and we proceed by induction.
\end{proof}


\subsection{Properties of monomial orders}
\label{sec:monom}
In this section we will assume that $\cP$ is equipped with a monomial order.

\begin{lemma}\label{lm:C(S)2}
Let $M\in\cN_\bd$ and $S\sbs \cP$ be a tree.
If $v\in\cP$ is such that
$m_u\in M_{S_{\le u}}$, for all $u\in C(S)$ with $u\le v$,
then $m_v\in M_{S_{\le v}}$.
\end{lemma}
\begin{proof}
We can assume by induction that $m_w\in M_{S_{\le w}}$ for all $w<v$.
Let $v=au$ for some arrow $a$ and $u\in\cP$.
If $u\in S$, then $v\in S\cup C(S)$ and we are done.
If $u\notin S$, then $m_u\in M_{S_{<u}}$,
hence $m_u=\sum_{w\in S_{<u}}c_wm_w$ for some scalars $c_w$.
For any $w\in S_{<u}$, we have $aw<au=v$, hence $m_{aw}\in M_{S_{\le aw}}\sbs M_{S_{<v}}$.
We conclude that 
$m_v=am_u=\sum_{w\in S_{<u}}c_wm_{aw}\in M_{S_{<v}}$.
\end{proof}

\begin{corollary}
Let $M\in\cN_\bd$ and $S\sbs \cP$ be a tree such that
$m_v\in M_{S_{\le v}}$, for all $v\in C(S)$.
Then the same is true for all $v\in\cP$.
\end{corollary}

We will see in Lemma \ref{intersect1} that the cells $Z_S\sbs\cN_\bd$, for $S\in\cP(\bd)$, are disjoint.
The following result gives an algorithm that, for any $M\in\cN_\bd$, produces a tree $S\in\cP(\bd)$ such that $M\in Z_S$.
It implies (in the case of monomial orders) that $\cN_\bd$ is a disjoint union of the cells $Z_S$ (\cf Theorem \ref{th:cell-dec}, where this statement is proved for admissible orders).

\begin{theorem}
Given $M\in\cN_\bd$, consider the maximal sequence $(v_1,\dots,v_n)$ such that
$$v_k=\min\sets{v\in\cP}
{m_{v}\notin \angs{m_{v_i}}{i<k}},\qquad 1\le k\le n.$$
Then $S=(v_1<\dots<v_n)$ is a tree and $M\in Z_S$.
\end{theorem}
\begin{proof}
We have 
$m_{v_{k+1}}\notin \angs{m_{v_i}}{i<k}$,
hence $v_k<v_{k+1}$ by  the minimality of $v_k$.

Let us assume by induction that $S'=(v_1<\dots<v_k)$ is a tree.
If $u\in\cP$ is such that $m_u\notin M_{S'_{\le u}}$,
then $S'_{\le u}=S'$ and $m_u\notin M_{S'}$.
Indeed, otherwise $S'_{\le u}=(v_1<\dots<v_i)$ for some $i<k$.
But then $v_{i+1}\le u$ by the minimality of $v_{i+1}$.
Therefore $v_{i+1}\in S'_{\le u}$, a contradiction.

Let $v=v_{k+1}\in\cP$ be the minimal element such that $m_v\notin M_{S'}$.
For any $u<v$, we have $m_u\in M_{S'}$, hence $m_u\in M_{S'_{\le u}}$ by the previous statement.
If $v\notin C(S')$, then we conclude from Lemma \ref{lm:C(S)2} that $m_v\in M_{S'_{\le v}}$, a contradiction.
Therefore $v\in C(S')$ and $S'\cup\set{v}$ is a tree.

%

By the maximality of the sequence, we conclude that $m_v\in M_S$ for all $v\in\cP$, hence $(m_u)_{u\in S}$ is a basis of $M$.
If $M\notin Z_S$, then there exists an element $v\in \cP$ such that $m_v\notin M_{S_{\le v}}$.
But then $m_v\notin M_S$ by the statement we proved.
This is a contradiction.
\end{proof}

\subsection{Cell decomposition}
\label{sec:cell dec}
In this section we give slightly simplified proofs  of some results from \cite{engel_smooth,reineke_cohomology}.
As before, we assume that \cP is equipped with an admissible order.
We define a total order on the set of trees $\cP(\bd)$ as follows.
Given trees $S,S'\in\cP(\bd)$ we write them in the form
\begin{equation}\label{tree order}
S=(v_1<\dots<v_k),\qquad S'=(v'_1<\dots<v'_{k})
\end{equation}
and define $S'<S$ if, for the minimal $1\le i\le k$ with $v'_i\ne v_i$, we have $v'_i<v_i$.


\begin{lemma}\label{intersect1}
If $S'<S$, then $Z_S\cap U_{S'}=\es$.
In particular, if $S\ne S'$, then $Z_S\cap Z_{S'}=\es$.
\end{lemma}
\begin{proof}
If $M\in Z_S\cap U_{S'}$, then $k=\dim M=\# S=\# {S'}$.
Let
$$S=(v_1<\dots<v_k),\qquad
S'=(v'_1<\dots<v'_k)$$
and let $1\le i\le k$ be minimal such that $v_i\ne v'_i$.
Then $v'_i<v_i$ and $v'_i\notin S$ (if $v'_i\in S$, then $v'_i=v_j$ for some $j<i$, hence $v'_j<v'_i=v_j$, a contradiction).
Note that $v_1=v'_1=*$, hence $v'_i\ne *$.
Let $v'_i=av$, where $a$ is an arrow.
Then $v<v'_i$ and $v\in S'$, hence $v=v'_j=v_j\in S$ for some $j<i$.
This implies that $v'_i=av\in C(S)$.
By the assumption $M\in Z_S$, we have 
$m_{v'_i}\in M_{S_{<v'_i}}$.
If $u\in S$ and $u< v'_i<v_i$, then $u=v_j=v'_j\in S'$ for some $j<i$.
This implies $m_{v'_i}\in M_{S'_{<v'_i}}$.
A contradiction to $M\in U_{S'}$.
\end{proof}

\begin{lemma}\label{lm2}
For any tree $S\in\cP(\bd)$, we have $U_S\sbs Z_S\cup \rbr{\bigcup_{S'<S}U_{S'}}$.
\end{lemma}
\begin{proof}


We need to show that
$U_S\ms Z_S\sbs\bigcup_{S'<S}U_{S'}$.
Let $M\in U_S\ms Z_S$
and let $v\in C(S)$ be minimal such that $m_v\notin M_{S_{<v}}$.
Consider the tree 
$$\bar S=\sets{u\in S}{u<v}\cup\set v.$$
Then the vectors $(m_u)_{u\in \bar S}$ are linearly independent and there exists a tree $\bar S\sbs S'\sbs\cP$ such that $M\in U_{S'}$ by Corollary \ref{basis ext}.
We will prove that $S'<S$.
We can write 
$$\bar S=(v_1<\dots<v_{i-1}<v),\qquad
S=(v_1<\dots<v_{i-1}<v_i<\dots),$$
where $v<v_i$.
We claim that for all $w\in S'\ms \bar S$, we have $w>v$, hence
$$S'=(v_1<\dots<v_{i-1}<v<\dots)$$
and therefore $S'<S$.

If our claim is not true, choose the minimal $w\in S'\ms \bar S$ with $w<v$.
Then $w\in C(\bar S\ms\set v)$,
hence either $w\in S$ (but then $w\in \bar S$, a contradiction) or $w\in C(S)$.
We have $m_{w}\notin  M_{S'_{<w}}$ as $w\in S'$ and $M\in U_{S'}$.
Note that if $u\in S$ and $u<w<v$, then $u\in\bar S\sbs S'$.
Therefore $M_{S_{<w}}\sbs M_{S'_{<w}}$,
hence $m_{w}\notin M_{S_{<w}}$, contradicting to the minimality of~ $v$.
\end{proof}

We arrive at the following result (\cf \cite[Theorem 7.7]{reineke_cohomology}).

\begin{theorem}
\label{th:cell-dec}
For any tree $S\in\cP(\bd)$, define the open set $V_S=\bigcup_{S'\le S}U_{S'}$.
Then $V_S=\bigsqcup_{S'\le S}Z_{S'}$.
In particular, 
$$\cN_\bd=\bigsqcup_{S\in\cP(\bd)}Z_S.$$
\end{theorem}
\begin{proof}
Recall that $\cP(\bd)$ is a finite set.
We conclude from Lemma \ref{lm2} by induction
that $V_S\sbs\bigcup_{S'\le S}Z_S$.
This union is disjoint by Lemma \ref{intersect1}.
Moreover, for any $S'\le S$, we have $Z_{S'}\sbs U_{S'}\sbs V_S$. Therefore $V_S=\bigsqcup_{S'\le S}Z_{S'}$.
Taking the maximal tree $S\in\cP(\bd)$, we obtain the last statement.
\end{proof}

\begin{corollary}[Cell decomposition]
The moduli space $\cN_\bd$ admits a cell decomposition, meaning that it admits a filtration by open subvarieties
$$\es=V_0\sbs V_1\sbs\dots\sbs V_r=\cN_\bd$$
such that $V_k\ms V_{k-1}$ are affine spaces, for all $1\le k\le r$.
\end{corollary}
\begin{proof}
Let us order the trees in $\cP(\bd)$
as $S_1<\dots<S_r$.
Then $V_k=V_{S_k}=\bigcup_{i\le k}U_{S_i}
=\bigsqcup_{i\le k}Z_{S_i}$ are open and satisfy
$V_k\ms V_{k-1}=Z_{S_k}$.
\end{proof}

\subsection{Examples of cell closures}
\label{sec:closures}
\begin{example}
For any $S\in\cP(\bd)$, the subset 
$\bigsqcup_{S''\ge S}Z_{S''}\sbs\cN_\bd$ is closed.
Let us show that there exist trees $S'>S$ such that $d(S')>d(S)$.
This will imply, in particular, that $\bigsqcup_{S''\ge S}Z_{S''}$ is not equal to the closure $\bar Z_S$.
We consider the framed quiver $Q^\bw$
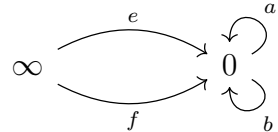
\begin{ctikzcd}
\infty\ar[rr,bend left,"e"]\ar[rr,bend right,"f"']&& 0
\ar[loop, out=30,in=90,looseness=5,"a"']
\ar[loop, out=-30,in=-90,looseness=5,"b"]
\end{ctikzcd}
with the order of arrows $e<f<a<b$ and the corresponding shortlex order on the set of paths ~$\cP$.
Let us consider the trees in $\cP(3)$ (we omit the root $*$)
$$
S=(e,f,bf),\qquad
S'=(e,ae,be)
$$
satisfying $S<S'$.
Then $d(S)=12$ and $d(S')=13$,
hence $Z_{S'}$ is not contained in $\bar Z_S$.
\end{example}

\begin{example}
Let us show that the closure of a cell is not necessary a disjoint union of cells.
Let us consider the framed quiver $Q^\bw$
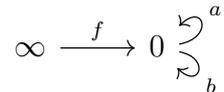
$$\begin{tikzcd}
\infty\ar[r,"f"]& 0
\ar[loop, out=10,in=45,looseness=5,"a"']
\ar[loop, out=-10,in=-45,looseness=5,"b"]
\end{tikzcd}
$$
with the order of arrows $f<a<b$ and the corresponding shortlex order on the set of paths ~\cP.
We consider the trees in $\cP(4)$ (we omit the root $*$)
$$S=(f, af, bf, b^2f),\qquad
S'=(f, bf, abf, b^2f)$$
satisfying $S<S'$.
Their critical sets are
$$C(S)=(a^2f,baf,abf,ab^2f,b^3f),\qquad
C(S')=(af,a^2bf,babf,ab^2f,b^3f).$$
The corresponding cells have dimensions
$d(S)=d(S')=17$.
We will show that $\bar Z_S\cap Z_{S'}\ne\es$.
On the other hand, $Z_{S'}\not\sbs \bar Z_S$ for dimension reasons, hence $\bar Z_S$ is not a disjoint union of cells.

To find $\bar Z_S\cap Z_{S'}$, we first determine 
$\bar Z_S\cap U_{S'}$ and then intersect it with $Z_{S'}$.
The intersection $\bar Z_S\cap U_{S'}$ is equal to the closure of $Z_S\cap U_{S'}\ne\es$ in $U_{S'}$.
Let us denote the elements of $S'$ by $u_1<\dots<u_4$ and the elements of $C(S')$ by $v_1<\dots<v_5$.
Then the coordinates $(c_{ij})$ of $M\in U_{S'}$ satisfy $m_{v_j}=\sum_i c_{ij}m_{u_i}$.
For $M\in Z_{S}$, we require that
$(m_f,m_{af},m_{bf},m_{b^2f})$ are linearly independent
and that
$m_{a^2f},m_{baf},m_{abf}$ are contained in $V=\ang{m_f,m_{af},m_{bf}}$.
The first condition means that for
$$m_{af}
=c_{11}m_f+c_{21}m_{bf}+c_{31}m_{abf}+c_{41}m_{b^2f}$$
we have $c_{31}\ne0$.
Then condition $m_{abf}\in V$ means that $c_{41}=0$,
hence $V=\ang{m_f,m_{bf},m_{abf}}$.
We have
$$m_{a^2f}=am_{af}
=c_{11}m_{af}+c_{21}m_{abf}+c_{31}m_{a^2bf}
\in V+c_{31}c_{42}m_{b^2f},$$
$$m_{baf}=bm_{af}
=c_{11}m_{bf}+c_{21}m_{b^2f}+c_{31}m_{babf}
\in V+(c_{21}+c_{31}c_{43})m_{b^2f}.
$$
Therefore $Z_S\cap U_{S'}$ is given by $c_{31}\ne0$ and
the equations
$$c_{41}=c_{42}=c_{21}+c_{31}c_{43}=0$$
while $\bar Z_S\cap U_{S'}$ is given by the above equations.
The cell $Z_{S'}\sbs U_{S'}$ is given by the equations $c_{21}=c_{31}=c_{41}=0$.
We conclude that $\bar Z_S\cap Z_{S'}$ is given by the equations $c_{21}=c_{31}=c_{41}=c_{42}=0$.
In particular, it is non-empty.
\end{example}

\subsection{Degeneracy loci}
For any subset $S\sbs\cP$, we define the degeneracy locus
\begin{equation}\label{deg loc1}
D(S)\sbs\cN_\bd
\end{equation}
to be the set of representations $M\in\cN_\bd$ such that the vectors $(m_u)_{u\in S}$ in $M$ are linearly dependent.
For any tree $S\in\cP(\bd)$ and $v\in C(S)$, let
\begin{equation}\label{Sv}
S(v)=\sets{u\in S_{t(v)}}{u<v}\cup\set v.
\end{equation}
Define the degeneracy locus of the tree $S$ to be
\begin{equation}\label{tree deg locus}
D_S=\bigcap_{v\in C(S)}D(S(v)).
\end{equation}

\begin{lemma}\label{D_S inclusion}
We have $D_S\sbs\bigsqcup_{S'\ge S}Z_{S'}$
and
$D_S\cap U_S=Z_S$.
\end{lemma}
\begin{proof}
We note that $\bigsqcup_{S'<S}Z_{S'}=\bigcup_{S'<S}U_{S'}$.
Therefore to show that $D_S\sbs\bigsqcup_{S'\ge S}Z_{S'}$ it is enough to show that $D_S\cap U_{S'}=\es$ for $S'<S$.
The proof of this statement goes through the same lines as the proof of Lemma \ref{intersect1}.


We have $Z_S\sbs D_S\cap U_S$ by definition.
If $M\in D_S\cap U_S$, then $m_v\in M_{S_{<v}}$, for all $v\in C(S)$, hence $M\in Z_S$.
Therefore $D_S\cap U_S=Z_S$.
\end{proof}

\begin{remark}
If $Q$ is a quiver with one vertex, then one can show that $D_S=\bigsqcup_{S'\ge S}Z_{S'}$
\cite[Lemma 2.2]{franzen_cohomology}.
We don't know if this statement is true in general.
\end{remark}

\subsection{Varieties having cell decomposition}
The following result is folklore.

\begin{theorem}\label{th:cell-dec-BM}
Assume that an algebraic variety $X$ admits a cell decomposition, meaning that it admits a filtration by open subvarieties
$$\es=V_0\sbs V_1\sbs\dots\sbs V_r=X$$
such that $Z_k=V_k\ms V_{k-1}$ are affine spaces, for all $1\le k\le r$.
Then
$$A_*(X)=\bop_{k}\bZ[\bar Z_k].$$
Moreover, $\HB_i(X,\bZ)=0$ for odd $i$ and the cycle map $A_i(X) \to \HB_{2i}(X,\bZ)$ is an isomorphism for $i\in\bZ$.
\end{theorem}
\begin{proof}
We have $V_k=\bigsqcup_{i\le k}Z_i$.
Define closed subsets 
$D_k=\bigsqcup_{i\ge k}Z_i=X\ms V_{k-1}$.
Then $Z_k$ is open in $D_k$ and its complement is $D_{k+1}$.
We obtain a long exact sequence of BM homology
$$\dots
\to \HB_i(D_{k+1})\to \HB_i(D_k)\to \HB_i(Z_k)\to\dots.$$
Note that $\HB_i(Z_k)=0$ for odd $i$ and we obtain by induction (on decreasing $k$) that $\HB_i(D_k)=0$ for odd $i$.
This implies that we have a commutative diagram with exact rows
\begin{ctikzcd}
&A_i(D_{k+1})\rar\dar&A_i(D_k)\rar\dar&A_i(Z_k)\dar\rar&0\\
0\rar&\HB_{2i}(D_{k+1})\rar&\HB_{2i}(D_k)\rar
& \HB_{2i}(Z_k)\rar&0.
\end{ctikzcd}
The right vertical map is an isomorphism as $Z_k$ is an affine space.
We can assume by induction (on decreasing $k$) that the left vertical map is also an isomorphism.
Then by the snake lemma we conclude that the middle vertical arrow is also an isomorphism.
Therefore we have an exact sequence
$$0\to A_*(D_{k+1})\to A_*(D_k)\to A_*(Z_k)\to0.$$
Note that $A_*(Z_k)=\bZ[Z_k]$ and the above exact sequence splits canonically, where we map $[Z_k]\in A_*(Z_k)$ to $[\bar Z_k]\in A_*(D_k)$.
Therefore we obtain
$$A_*(D_k)\iso \bZ[\bar Z_k]\oplus A_*(D_{k+1}).$$
By induction this implies
$A_*(X)=A_*(D_1)\iso\bop_{k=1}^r\bZ[\bar Z_k].$
\end{proof}

\section{Multi-partitions}

\subsection{Trees and multi-partitions}
Let, as before, $\f\in\bN^I$ be a framing vector,
$Q^\f$ be the corresponding framed quiver, and $\cP=\cP^\f$
be the poset of paths in $Q^\f$ that start at $\infty$.
For $\bd\in\bN^I$, let $\cP(\bd)=\cP^\f(\bd)$ be the set of trees $S\sbs\cP$ with $\udim S=\bd$.
We define
\begin{equation}
c(\bd)=\f-\hi(\bd,-)\in\bZ^I,\qquad
c(\bd)_i=\f_i-\hi(\bd,e_i),\qquad i\in I,
\end{equation}
so that $\udim C(S)=c(\bd)$, for any $S\in\cP(\bd)$ 
(see Lemma \ref{lm:dim CS}).

We define $\cS(\bd)=\cS^\f(\bd)$ to be the set of multi-partitions
\begin{equation}\label{eq:multi-part1}
\la=(\tp\la i)_{i\in I},\qquad
\tp\la i
=(\tp\la i_1\ge\dots\ge\tp\la i_{\bd_i}\ge0),
\qquad i\in I,
\end{equation}
such that, 
for all $0\le \be<\bd$ (meaning that $\be\ne\bd$ and $0\le\be_i\le\bd_i$ for all $i\in I$),
there exists $i\in I$ satisfying
(we define $\tp\la i_{0}=+\infty$)
\begin{equation}\label{Phi}
\la^{(i)}_{\bd_i-\be_i}<c(\be)_i.
\end{equation}

We assume that $\cP$ is equipped with an admissible total order.
The following result was formulated in \cite{engel_smooth}, although some details of the proof were omitted.

\begin{theorem}
\label{bijection}
There is a bijection $\cP(\bd)\to\cS(\bd)$ that sends a tree $S\in\cP(\bd)$ with 
\begin{equation}
S_i=(u_{i,1}<\dots<u_{i,\bd_i}),\qquad i\in I,
\end{equation}
to the multi-partition $\la$ with
\begin{equation}\label{la from S}
\la^{(i)}_{\bd_i-k}=\#\sets{v\in C(S)_i}{v<u_{i,k+1}},
\qquad i\in I,\,0\le k<\bd_i.
\end{equation}
\end{theorem}

\begin{remark}
Given a tree $S\in\cP(\bd)$ and $v\in C(S)_i$,
we define \eqref{dim Z_S}
\begin{equation}
k_v=\#\sets{u\in S_{i}}{u<v}\le \bd_i.
\end{equation}
Then $k_v\le k$ if and only if $v<u_{i,k+1}$.
Therefore 
\begin{equation}
\la^{(i)}_{\bd_i-k}=\#\sets{v\in C(S)_i}{k_v\le k},\qquad
i\in I,\, 0\le k<\bd_i.
\end{equation}
\end{remark}

\begin{proof}[\textit{Proof of Theorem \ref{bijection}}]
First, let us show
that the multi-partition \la defined above satisfies
the required condition~\eqref{Phi}.
For convenience, we define 
(using $\tp\la i_0=+\infty$)
\begin{equation}\label{m}
\tp mi_k=\tp\la i_{\bd_i-k},\qquad 0\le k\le \bd_i.
\end{equation}
Then condition \eqref{Phi} on $\la$ 
means that for all $0\le \be<\bd$, there exists $i\in I$ such that
\begin{equation*}
\tp m i_{\be_i}<c(\be)_i.
\end{equation*}
Assume that there exists $0\le\be<\bd$ such that 
$\tp mi_{\be_i}\ge c(\be)_i$ for all $i\in I$.
This means that
\begin{equation}\label{assum1}
\#\sets{v\in C(S)_i}{v<u_{i,\be_i+1}}\ge c(\be)_i,
\end{equation}
whenever $\be_i<\bd_i$.
Let $i\in I$ be such that $\be_i<\bd_i$ and $u_{i,\be_i+1}$ is minimal.
We have
$$
\#\sets{v\in C(S)_i}{v<u_{i,\be_i+1}}
=\#\sets{v\in S_i\cup C(S)_i}{v\le u_{i,\be_i+1}}
-\be_i-1
$$ 
Every $v\in S_i\cup C(S)_i$ satisfying $v\le u_{i,\be_i+1}$
can be written in the form
$v=a\bar v$ for some 
$\bar v\in S_j$, $j\in I\cup\set\infty$, and some arrow $a:j\to i$ in $Q^\f$.

If $j\in I$, then $\bar v<v\le u_{i,\be_i+1}\le u_{j,\be_j+1}$, whenever $\be_j<\bd_j$.
The number of such $\bar v\in S_j$ is bounded by $\be_j$ (this also includes the case $\be_j=\bd_j$).
Therefore the number of $v$ of the required form
is bounded by $\sum_j a_{ji}\be_j$, where $a_{ji}$ is the number of arrows $j\to i$ in $Q^\f$.
If $j=\infty$, then $\bar v=e_\infty$ and the number of $v$ of the required form is bounded by $\f_i$.
We conclude that
$$\#\sets{v\in C(S)_i}{v<u_{i,\be_i+1}}
\le \f_i+\sum_j a_{ji}\be_j-\be_i-1
=\f_i-\hi(\be,e_i)-1<c(\be)_i$$
and this contradicts assumption \eqref{assum1}.

Now let us construct the inverse map $\cS(\bd)\to\cP(\bd)$.
Given a multi-partition $\la\in\cS(\bd)$, we define the numbers $\tp mi_k$ as in \eqref{m}.
We need to construct a tree 
$S=(u_0<u_1<\dots)$
in $\cP(\bd)$ with
\begin{equation*}
S_i=(u_{i,1}<\dots<u_{i,\bd_i}),\qquad i\in I,
\end{equation*}
such that
$$\#\sets{v\in C(S)_i}{v<u_{i,k+1}}=\tp m i_k,
\qquad 0\le k<\bd_i.$$
We construct this tree inductively as follows.
We set $u_0=e_\infty$.
Assume that a tree $S'=(u_0<\dots<u_{r})$
is constructed
and satisfies the conditions
\begin{equation}\label{cond1}
\#\sets{v\in C(S')_i}{v<u_{i,k+1}}=\tp m i_{k},
\qquad i\in I,\, 0\le k<\be_i.
\end{equation}
\begin{equation}\label{cond2}
\#\sets{v\in C(S')_i}{v<u}\le \tp mi_{\be_i}\qquad i\in I,\,u\in S'.
\end{equation}
where $\be=\udim S'\le\bd$ and
$$S'_i=(u_{i,1}<\dots<u_{i,\be_i}),\qquad i\in I.$$
Assuming that $\be<\bd$, we will construct a tree $S''=S'\cup\set{u_{r+1}}$ satisfying the same conditions.

By the assumption on \la,
there exists $i\in I$ such that
$$\tp m i_{\be_i}<c(\be)_i.$$
We order the elements of $C(S')_i$ and denote the $\tp m i_{\be_i}+1$-st element by $v_{i}$.
We claim that $v_i$ is the maximal element in 
$S'\cup\set {v_i}$.
Indeed, we have 
$$\#\sets{v\in C(S')_i}{v\le v_i}=\tp mi_{\be_i}+1$$
while by condition \eqref{cond2}
$$\#\sets{v\in C(S')_i}{v\le u}\le \tp mi_{\be_i},\qquad u\in S'.$$
This implies that $u<v_i$ for all $u\in S'$.

Let $j\in I$ be such that $v_j\in C(S')_j$ is minimal
and let $S''=S'\cup\set{v_j}$.
Then
$$C(S'')=C(S')\ms\set{v_j}\cup\ch(v_j),$$
where $\ch(v)=\min_{\preceq}(\cP_{\succ v})$ denotes the set of children of $v\in\cP$.
As $v_j$ is the maximal element in $S''$, we have $S''_i=S'_i$ for $i\ne j$ and
$$S''_j=(u_{j,1}<\dots<u_{j,\be_j}<v_j=:u_{j,\be_j+1}).$$

We will show that conditions \eqref{cond1} and \eqref{cond2} are satisfied for the tree $S''$.
For $i\in I$ and $0\le k<\be_i$, we have
\begin{equation*}
\#\sets{v\in C(S'')_i}{v<u_{i,k+1}}
=\#\sets{v\in C(S')_i}{v<u_{i,k+1}}
=\tp m i_{k}
\end{equation*}
by condition \eqref{cond1} for $S'$.
On the other hand
$$\#\sets{v\in C(S'')_j}{v<v_j}
=\#\sets{v\in C(S')_j}{v<v_j}
=\tp mj_{\be_j+1}.$$
Therefore condition \eqref{cond1} for the tree $S''$ is satisfied.
To prove condition \eqref{cond2} for $S''$ it is enough to consider the maximal element $u=v_j$ of the tree.

If $i\in I$ is such that $\tp mi_{\be_i}<c(\be)_i$, then 
$v_j\le v_i$ by the minimality of $v_j$, hence
$$\#\sets{v\in C(S')_i}{v<v_j}\le \tp mi_{\be_i}.$$
If $i\in I$ is such that $\tp mi_{\be_i}\ge c(\be)_i$, then we also obtain
$$\#\sets{v\in C(S')_i}{v<v_j}\le 
\# C(S')_i=c(\be)_i\le\tp mi_{\be_i}.$$
We conclude that for all $i\in I$
$$\#\sets{v\in C(S'')_i}{v<v_j}
=\#\sets{v\in C(S')_i}{v<v_j}\le \tp mi_{\be_i}$$
which implies condition \eqref{cond2} for the tree $S''$.
\end{proof}

\subsection{Total order on the set of multi-partitions}
In the previous section we constructed a bijection $\cP(\bd)\iso\cS(\bd)$ which depends on an admissible order on $\cP$.
On the other hand, 
in ~\S\ref{sec:cell dec} we constructed a total order on the set of trees $\cP(\bd)$.
This implies that we have the induced total order on the set of multi-partitions $\cS(\bd)$, which a priori may depend on an admissible total order on \cP.
In this section we will show that, for a quiver $Q$ having only one vertex, the induced total order on $\cS(\bd)$ is independent of an admissible order on \cP.

Let us assume that $Q$ has one vertex.
Recall that,
for $d\in\bN\iso\bN^I$ and a tree 
$S=(*<u_1<\dots<u_d)$ in $\cP(d)$,
the corresponding partition $\la=(\la_1,\dots,\la_d)\in\cS(d)$ 
is defined by \eqref{la from S}
\begin{equation}
\la_{d-k}=\#\sets{v\in C(S)}{v<u_{k+1}},\qquad 0\le k<d.
\end{equation}

\begin{lemma}\label{lm:part-order}
Let $S,S'\in\cP(d)$ and $\la,\mu\in\cS(d)$ be the corresponding partitions.
Then $S<S'$ if and only if $\la_k<\mu_k$ for the maximal $1\le k\le d$ with $\la_k\ne\mu_k$.
In particular, the induced total order on $\cS(d)$ is independent of the choice of an admissible order on \cP.
\end{lemma}
\begin{proof}
It is enough to show that if partitions $\la,\mu$ satisfy the above property, then $S<S'$.
Let $m_k=\la_{d-k}$ and $m'_k=\mu_{d-k}$ for $0\le k<d$.
By the proof of Theorem \ref{bijection} the tree $S=(*<u_1<\dots<u_d)$ has the property that $u_{k+1}$ is the $(m_k+1)$-st element in $C(\bar S)$, where $\bar S$ is the tree  $\bar S=(*<u_1<\dots<u_k)$.
Similarly for the tree $S'=(*<u'_1<\dots<u'_d)$.

By assumption, we have $m_{k+1}<m'_{k+1}$ for the minimal $k$ with $m_{k+1}\ne m'_{k+1}$.
This implies that $\bar S=(*<u_1<\dots<u_{k})=(*<u'_1<\dots<u'_k)$
and $u_{k+1}<u'_{k+1}$.
Therefore $S<S'$.
\end{proof}

\subsection{Cells parametrized by multi-partitions}
Theorem \ref{bijection} implies that we can parametrize the cells $Z_S\sbs\cN_\bd$, for $S\in\cP(\bd)$, by multi-partitions.
Note that the number of $v\in C(S)_i$ with 
$$k_v=\#\sets{u\in S_i}{u<v}=\bd_i-k$$
is equal to $\la^{(i)}_k-\la^{(i)}_{k+1}$ for $k\ge0$,
where we use now 
$\la^{(i)}_0=\# C(S)_i=\f_i-\hi(\bd,i)$.
Therefore, the dimension of the cell $Z_S$ is
\begin{equation}\label{d(la)}
\dim Z_S=\sum_{v\in C(S)}k_v
=\sum_{i\in I}\sum_{k\ge0}(\bd_i-k)(\la^{(i)}_k-\la^{(i)}_{k+1})
=\f\cdot\bd-\hi(\bd,\bd)-\n\la,
\end{equation}
where $\n\la=\sum_{i\in I}\n{\la^{(i)}}
=\sum_{i\in I}\sum_{k\ge1}{\la^{(i)}_k}$.
This implies that the motivic class of $\cN_\bd$ can be written in the form
\cite[Theorem 6.2]{engel_smooth}
\begin{equation}
[\cN_\bd]
=\bL^{\f\cdot\bd-\hi(\bd,\bd)}
\sum_{\la\in\cS(\bd)}\bL^{-\n\la}.
\end{equation}

It follows from Theorem \ref{bijection} that,
for an admissible total order on $\cP$,
we have a bijection $\cS(\bd)\to\cP(\bd)$, $\la\mto S_\la$.
We will see later that the class $[\bar Z_{S_\la}]\in A_*(\cN_\bd)$ generally depends on the choice of an admissible order on \cP.



\section{Grassmannians and non-commutative Hilbert schemes}
Let $0\le d\le w$ be two integers, $r=w-d$ and $W=\bC^w$.
We define the Grassmannian
\begin{equation}
\Gr^d(W)=\Gr_r(W)=\Gr(r,w)
\end{equation}
to be the algebraic variety parameterizing
surjective linear maps $M:W\to\bC^d$ (up to the action of $\GL_d(\bC)$) or, equivalently,
subspaces $U=\ker(M)\sbs W$ of dimension $r$.

We can interpret it as a non-commutative Hilbert scheme as follows.
Let $Q$ be the quiver having one vertex $0$ and no arrows,
and let $Q^w$ be the framed quiver obtained from $Q$ by adding one vertex $\infty$ and $w$ arrows $\infty\to 0$.
The non-commutative Hilbert scheme $\cN_d=\cN_{d,w}$ parametrizes representations $M$ of $Q^w$ having dimension vector $(d,1)$ (with the coordinate $1$ at the vertex~$\infty$), generated by $M_\infty$.
Such representations can be identified with surjective linear maps $M:\bC^w\to\bC^d$ 
(up to the action of $\GL_d(\bC)$), hence
\begin{equation}
\cN_{d}=\Gr^d(W)=\Gr_{r}(W).
\end{equation}
This is a smooth projective variety of dimension $d(w-d)$.

In the previous sections we established a cell decomposition of $\cN_d$ parametrized by trees or partitions.
On the other hand, Grassmannians possess a decomposition by Schubert cells parametrized by partitions.
In this section we will establish a precise relationship between these decompositions.
We will also discuss if the properties satisfied by cell decompositions of Grassmannians can be generalized to other quivers.

\subsection{Trees and cells}
Recall that the framed quiver $Q^w$ consists of two vertices $\infty,\, 0$ and $w$ arrows $\infty\to 0$.
Let us define $[1,w]=\set{1,\dots,w}$.
Then
\begin{enumerate}
\item 
The path poset $\cP$ can be identified with the set $\set{*}\cup[1,w]$.
\item
A subtree $S\in\cP(d)$ can be identified with a subset $S_0\sbs[1,w]$ having $d$ elements.
\item
The critical set $C(S)$ can be identified with $S'=[1,w]\ms S_0$ having $r=w-d$ elements.
\end{enumerate}

In what follows we will write $S$ for $S_0\sbs[1,w]$ and still call it a tree.
We order the elements of $\cP$ as
$$*<1<\dots<w.$$
A framed representation $M\in\cN_d$
can be identified with a surjective linear map $M:\bC^w\to\bC^d$.
We define 
$$m_i=M(e_i)\in\bC^d,\qquad i\in[1,w].$$
Then the cell $Z_S$, for $S\sbs[1,w]$, consists of such $M$ that
\begin{enumerate}
\item $(m_i)_{i\in S}$ are linearly independent.
\item $m_j\in\angs{m_i}{i\in S_{<j}}$ for all $j\in S'=[1,w]\ms S$.
\end{enumerate}

Earlier we defined the set $\cS(d)=\cS^w(d)$ to be the set of partitions \eqref{eq:multi-part1}
\begin{equation}\label{part2}
\mu=(\mu_1,\dots,\mu_d),\qquad \mu_1\le w-d.
\end{equation}
In Theorem \ref{bijection} we described a bijection
$\cP(d)\to\cS(d)$.
For every tree
$$S=(u_1<\dots<u_d)\sbs[1,w]$$
we define the corresponding partition $\mu\in\cS(d)$ by
\begin{equation}\label{mu}
\mu_{d-k}=\#\sets{j\in [1,w]\ms S}{j<u_{k+1}},
\qquad 0\le k<d.
\end{equation}
We denote the cell $Z_S$ also by $Z_\mu$.

%

\subsection{Schubert cells}
As before, we consider $0\le d\le w$ and $r=w-d$.
Consider a full flag
$$0=F_0\sbs F_1\sbs\dots\sbs F_w=W=\bC^w.$$
Given a partition 
\begin{equation}\label{part3}
\la=(\la_1,\dots,\la_{r}),\qquad \la_1\le \la_0:=d,
\end{equation}
we define 
the Schubert variety \fulb[\S9.4]
\begin{equation}
\Om_\la
=\sets{U\in\Gr^d(W)}{\dim(U\cap F_{d+i-\la_i})\ge i,\,1\le i\le r}
\end{equation}
and the Schubert cell 
\begin{equation}
\Om_\la^\circ
=\sets{U\in\Gr^d(W)}{\dim(U\cap F_k)=i\text{ for }
d+i-\la_i\le k\le d+i-\la_{i+1},\,0\le i\le r}.
\end{equation}
It is known that
\begin{enumerate}
\item  $\Om_\la^\circ$ is an affine space.
\item $\Gr^d(W)=\bigsqcup_\la\Om_\la^\circ$.
\item $\Om_\la$ is the closure of $\Om_\la^\circ$ in $\Gr_r(W)$.
\item $\Om_\la=\bigsqcup_{\nu\supseteq\la}\Om_\nu^\circ$,
where $\nu\supseteq\la$ if $\nu_i\ge\la_i$ for all $i$.
\end{enumerate}

\begin{proposition}\label{prop:cells}
Assume that the flag on $\bC^w$ is given by $F_i=\ang{e_1,\dots,e_i}$, for $1\le i\le w$.
Then, for any partition \la as in \eqref{part3}, we have
\begin{equation}
\Om_\la^\circ= Z_S,
\end{equation}
where $S=[1,w]\ms S'$ with
\begin{equation}
S'=\set{v_1,\dots,v_r},\qquad
v_i=d+i-\la_i,\qquad 1\le i\le r.
\end{equation}
\end{proposition}
\begin{proof}
We have
\begin{equation*}
0=v_0<v_1<\dots<v_r\le d+r=w
\end{equation*}
so that $S'$ consists of $r$ elements and $S$ consists of $d$ elements.

Given $U\sbs\Gr_r(\bC^w)$, we consider the quotient
$$M:\bC^w\to \bC^w/U\iso\bC^d.$$
Then the condition $\dim U\cap F_k=i$, means that $\dim\ker (M:F_k\to\bC^d)=i$, hence
\begin{equation}\label{rk1}
\rk\ang{m_1,\dots,m_k}=k-i,\qquad m_i=M(e_i)\in\bC^d.
\end{equation}
The condition that $U\in\Om_\la^\circ$ means that we have the above equality for $v_i\le k<v_{i+1}$.
This means that the rank in \eqref{rk1} increases by $1$ if $v_i<k<v_{i+1}$ (hence $k\in S$) and doesn't increase if $k=v_{i+1}$ (hence $k\in S'$).
This can be reformulated as the condition that
$(m_i)_{i\in S}$ are linearly independent and that
$m_j\in\angs{m_i}{i\in S_{<j}}$ for all $j\in S'$.
This is exactly the definition of the cell $Z_S$.
\end{proof}

\subsection{Relation between partitions}
In Prop.~\ref{prop:cells}, for any partition $\la$ as in \eqref{part3}, we constructed the sets
$$S=(u_1<\dots< u_d)\sbs[1,w],\qquad S'=[1,w]\ms S=(v_1<\dots<v_r),$$
with $v_i=d+i-\la_i$.
Conversely, we can reconstruct the partition $\la$ from $S$ by
$$\la_i=d+i-v_i,\qquad 1\le i\le r.$$
On the other hand, we associated with $S$ a partition $\mu$ with \eqref{mu}
$$\mu_k=\#\sets{i\ge1}{v_i<u_{d-k+1}},\qquad 1\le k\le d,$$

\begin{lemma}
The partitions $\mu$ and $\la$ are conjugate to each other, meaning that
$$\mu_k=\#\sets{i\ge1}{\la_i\ge k},\qquad k\ge1.$$
\end{lemma}
\begin{proof}
We need to show that
$$\la_i\ge k\iff v_i<u_{d-k+1}.$$
The first inequality can be written in the form
$v_i\le d-k+i.$
In this case the number of $j$ with $u_j<v_i$ is $\le d-k$, implying that $v_i<u_{d-k+1}$.

If $v_i<u_{d-k+1}$, then the number of $j$ with $u_j<v_i$ is $\le d-k$.
Therefore $v_i$ is bounded by $d-k$ (the number of $u_j$ on the left of $v_i$) plus $i$ (the number of $v_j$ on the left of $v_i$), hence $v_i\le d-k+i$.
\end{proof}
 
We conclude from the previous lemma that
\begin{equation}
Z_\mu=\Om_{\mu'}^\circ,\qquad
\bar Z_\mu=\Om_{\mu'}.
\end{equation}
where $\mu'$ denotes the conjugate partition.
Schubert varieties have a decomposition
\begin{equation}
\Om_\la=\bigsqcup_{\nu\supseteq\la}\Om_\nu^\circ,
\end{equation}
where $\nu\supseteq\la$ if $\nu_i\ge\la_i$ for all $i$.
This implies that
\begin{equation}
\bar Z_\mu=\bigsqcup_{\nu\supseteq\mu}Z_\nu.
\end{equation}
On the other hand, we have seen in \S\ref{sec:closures} that the closure of a cell is not a disjoint union of cells for more general quivers.
The class of a Schubert variety $[\Om_{\mu'}]\in A_*(\Gr^d(W))$ is independent of a flag.
Therefore the class $[\bar Z_\mu]$ is independent of the choice of a total order on~$\cP$.
On the other hand, we will see in \S\ref{sec:dependence} that this is not the case for more general quivers.

\section{Coha modules}

\subsection{Cohomological Hall algebras}
\label{ss:CoHA}
Cohomological Hall algebras were introduced in \cite{kontsevich_cohomological}. 
They are defined for a quiver with a potential and use cohomology of vanishing cycle complexes
associated with the trace of the potential.
We restrict in this paper to the case where the potential is trivial.
In this case, way may work with ordinary equivariant cohomology groups.

Let $Q$ be a quiver with the set of vertices $I$
and the set of arrows $Q_1$.
We call elements $\bd \in\bN^I\sbs\Ga:=\bZ^I$ the dimension vectors.
As before, we define
$$
R_\bd = R(Q,\bd) = \bigoplus_{(a:i\to j) \in Q_1} \Hom(\bC^{\bd_i},\bC^{\bd_j}).
$$
It is a finite-dimensional complex vector space.
We regard an element of $R_\bd$ as a representation of $Q$ of dimension vector $\bd$. We consider the complex linear algebraic group $G_\bd = \prod_{i \in I} \GL_{d_i}(\bC)$ and let an element $g = (g_i)_{i \in I} \in G_\bd$ act on $M = (M_a)_{a \in Q_1} \in R_\bd$ by
$$
	g \cdot M := (g_jM_ag_i^{-1})_{(a:i \to j) \in Q_1}.
$$
Two elements of $R_\bd$ are isomorphic as representations of $Q$ if and only if they lie in the same $G_\bd$-orbit. The $\bC$-valued points of the quotient stack $[R_\bd/G_\bd]$ therefore correspond to isomorphism classes of $Q$-representations of dimension vector $\bd$.

Let $\be \in \bN^I$ be another dimension vector. For every $i \in I$, we identify $\bC^{\bd_i}$ with the subspace of $\bC^{\bd_i+\be_i}$ spanned by the first $\bd_i$ coordinate vectors. Let
\begin{equation}
R_{\bd,\be} =
\sets{M\in R_{\bd+\be}}
{M_a(\bC^{\bd_{i}})\sbs\bC^{\bd_{j}}\ \forall a:i\to j}.
\end{equation}
It is a linear subspace of $R_{\bd+\be}$. Elements of $R_{\bd,\be}$ are representations $M \in R_{\bd+\be}$ for which $M_a$, $a:i\to j$, has the following block upper triangular shape:
$$
	\begin{blockarray}{ccc} 
		& \bC^{\bd_{i}} & \bC^{\be_{i}} \\
		\begin{block}{c(cc)}
			\bC^{\bd_{j}} & M'_a & * \\
			\bC^{\be_{j}} & 0 & M''_a \\
		\end{block}
	\end{blockarray}
$$
For such $M \in R_{\bd,\be}$, we obtain representations $M' \in R_\bd$ and $M'' \in R_\be$. This gives rise to the maps
$$
	R_\bd \times R_\be \xleftarrow[]{p} R_{\bd,\be} \xrightarrow[]{i} R_{\bd+\be}.
$$
Note that $i$ is a closed embedding of codimension 
$r_1 := \sum_{a:i\to j} \bd_{i}\be_{j}$.
Let
$$
G_{\bd,\be} =
\sets{g \in G_{\bd+\be}}
{g_i(\bC^{\bd_i})\sbs\bC^{\bd_i}\ \forall i\in I},
$$
i.e.\ for $g \in G_{\bd,\be}$, the matrix $g_i$ has a similar block upper triangular form to the one from above. 
The group $G_{\bd,\be}$ is a parabolic subgroup of $G_{\bd+\be}$
and the dimension of the homogeneous space $G_{\bd+\be}/G_{\bd,\be}$ is $r_2 = \sum_{i \in I} \bd_i\be_i$. 
The subvariety $R_{\bd,\be}$ is $G_{\bd,\be}$-invariant. We obtain morphisms of algebraic groups
$$
	G_\bd \times G_\be \xleftarrow[]{q} G_{\bd,\be} \xrightarrow[]{j} G_{\bd+\be}.
$$
The map $p$ is $G_{\bd,\be}$-equivariant with respect to $q$ and $i$ is $G_{\bd,\be}$-equivariant with respect to $j$.

We now pass to equivariant cohomology. For a smooth complex variety $X$, equipped with an action of a complex algebraic group $G$, we consider both $X$ and $G$ with the euclidean topology. We form its equivariant singular cohomology $H_G^*(X;\Q)$ and regard it as a $\Z$-graded $\Q$-vector space. In the following, we will always work with rational coefficients and therefore write $H_G^*(X)$ for $H_G^*(X;\Q)$. 

The maps $p$, $i$, $q$, and $j$ give rise to linear maps in equivariant cohomology as follows:
$$
	\begin{tikzcd}[column sep=small]
		H_{G_\bd}^*(R_\bd) \otimes H_{G_\be}^*(R_\be) \arrow{d}{\iso} & & & H_{G_{\bd+\be}}^*(R_{\bd+\be})[2r_1-2r_2]\\
		H_{G_\bd \times G_\be}^*(R_\bd \times R_\be) \arrow{r}{q^*}[swap]{\iso} & H_{G_{\bd,\be}}^*(R_\bd \times R_\be) \arrow{r}{p^*}[swap]{\iso} & H_{G_{\bd,\be}}^*(R_{\bd,\be}) \arrow{r}{i_*} & H_{G_{\bd,\be}}^*(R_{\bd+\be})[2r_1] \arrow{u}{j_*}
	\end{tikzcd}
$$
Note that under the identification $H_{G_{\bd,\be}}^*(R_{\bd+\be}) \iso H_{G_{\bd+\be}}^*(G_{\bd+\be} \times^{G_{\bd,\be}} R_{\bd+\be})$, the map $j_*$ is the push-forward along $G_{\bd+\be} \times^{G_{\bd,\be}} R_{\bd,\be} \to R_{\bd+\be}$. Note also that $r_1-r_2 = -\chi(\bd,\be)$.

We define the $\Ga$-graded vector space
\begin{equation}
\cH = \bigoplus_{\bd \in \bN^I} \cH_\bd,\qquad
\cH_\bd = H_{G_\bd}^*(R_\bd).
\end{equation}
The composition of the above maps gives a morphism 
$*: \cH_\bd \otimes \cH_\be \to \cH_{\bd+\be}$ of $\bQ$-vector spaces and the induced morphism
$$*:\cH\ts\cH\to\cH$$
of \Ga-graded vector spaces.

\begin{theorem}[{\cite[\S2.2]{kontsevich_cohomological}}]
The map $*:\cH\ts\cH\to\cH$ defines the structure of an associative (unital)
$\Gamma$-graded algebra on \cH.
\end{theorem}

The algebra $\cH$ is called the \idef{cohomological Hall algebra} (\coha) of the quiver $Q$.

\begin{remark}
Assume that $Q$ is a symmetric quiver,
meaning that $\chi=\chi_Q$ is a symmetric bilinear form.
In this case, we may refine the construction as follows. Define the \bZ-graded vector space
$$
	\cH_\bd := H_{G_\bd}^*(R_\bd)[-\chi(\bd,\bd)].
$$
With this degree shift, the linear map  
$*:\cH_\bd \otimes \cH_\be \to \cH_{\bd+\be}$
is homogenous with respect to the $\bZ$-grading; here symmetry of the Euler form is crucial. 
This means that $\mathcal{H}$ becomes a $\Gamma$-graded algebra in the category of $\bZ$-graded vector spaces.
\end{remark}


\subsection{Shuffle algebra description}
Let us recall an explicit shuffle product description of the multiplication of \coha
from \cite{kontsevich_cohomological}.
For any $\bd\in\bN^I$, the group $G_\bd$ is linearly reductive.
The subgroup $T_\bd\sbs G_\bd$ consisting of tuples of invertible diagonal matrices is a maximal torus. The corresponding Weyl group is isomorphic to the product of symmetric groups $\Sigma_\bd = \prod_{i \in I} \Sigma_{\bd_i}$. As the space $R_\bd$ is 
$G_\bd$-equivariantly contractible, we obtain
\begin{equation}\label{coha as poly}
\mathcal{H}_\bd = H_{G_\bd}^*(R_\bd) \iso H_{G_\bd}^*(\pt) \iso H_{T_\bd}^*(\pt)^{\Sigma_\bd} \iso \bigotimes_{i \in I} \Q[x_{i,1},\ldots,x_{i,\bd_i}]^{\Sigma_{\bd_i}}.
\end{equation}
The last isomorphism is provided by the isomorphism $H_{T_\bd}^*(\pt) \iso S(X^*(T_\bd) \otimes \Q)$, where $X^*(T_\bd)$ is the group of characters of $T_\bd$. This is the free abelian group generated by the characters $x_{i,k}: T_\bd \to \C^\times$ which selects the $k$\textsuperscript{th} diagonal entry from the $i$\textsuperscript{th} matrix. 

\begin{theorem}[{\cite[Thm.\ 2]{kontsevich_cohomological}}]
Let $f \in \mathcal{H}_\bd$ and $g \in \mathcal{H}_\be$. The product $f * g \in \mathcal{H}_{\bd+\be}$ is given by
\begin{equation}
f*g 
=\sum_{\si\in\Sh(\bd,\be)}\si\rbr{
f(x_{i,r})_{(i,r)\in I_\bd}
\cdot g(x_{j,\bd_j+s})_{(j,s)\in I_\be} \cdot 
\prod_{i,j \in I} \prod_{r=1}^{\bd_i} \prod_{s=1}^{\be_j}
(x_{j,\bd_j+s} - x_{i,r})^{-\chi(e_i,e_j)}
}
\end{equation}
where the sum runs over all $(\bd,\be)$-shuffles, meaning $\si\in\Si_{\bd+\be}$ satisfying
\begin{equation}
\si_i(1)<\dots<\si_i(\bd_i),\qquad
\si_i(\bd_i+1)<\dots<\si_i(\bd_i+\be_i)\qquad \forall i\in I,
\end{equation}
and the set $I_\bd$ is given by 
$I_\bd=\sets{(i,k)}{i\in I,\,1\le k\le\bd_i}$.
\end{theorem}

Note that in the above theorem, $e_i \in \mathbb{N}^I$ is the simple root at the vertex $i$, while $\be_i$ is the $i$\textsuperscript{th} entry of the dimension vector $\be$.

\subsection{Modules over the Cohomological Hall algebra}
Modules over the Cohomological Hall algebra that come from framed stable representations were introduced in \cite{soibelman_remarks}. We briefly recall the construction.
Given a framing vector $\bw\in\bN^I$, we defined earlier the non-commutative Hilbert scheme
$$\Hilb_{\bd,\bw}=R^{\f,\st}_\bd/G_\bd,\qquad \bd\in\bN^I,$$
where $R^{\f,\st}_\bd\sbs R^{\f}_\bd$ is the open subvariety of stable framed representations in
$$R^\f_\bd
=R^\f(Q,\bd) = R(Q,\bd) \oplus \underbrace{\bop_{i\in I}\Hom(\bC^{\bw_i},\bC^{\bd_i})}_{=: F_{\bw,\bd}}.$$
We can write $M\in R^\f_\bd$ (uniquely) as $M = (N,f)$ with $N \in R(Q,\bd)$ and $f \in F_{\bw,\bd}$; sometimes we call $f$ the framing datum.

Let $\bd, \be$ be two dimension vectors of $Q$. 
We consider
\begin{equation}
R_{\bd,\be}^\f := R_{\bd,\be} \oplus F_{\bw,\bd+\be},
\qquad
R_{\bd,\be}^{\f,\st} := R_{\bd,\be}^\f \cap R_{\bd+\be}^{\f,\st}.
\end{equation}
That means, $R_{\bd,\be}^{\f,\st}$ consists of framed stable representations $M = (N,f)$ of dimension vector $\bd+\be$, such that $N$ lies in $R_{\bd,\be}$. We consider the maps 
$$R_\bd \times R_\be\xlto pR_{\bd,\be}\xto iR_{\bd+\be}$$
from \S\ref{ss:CoHA}. Let $M = (N,f) \in R_{\bd,\be}^{\f,\st}$. We have $p(N) = (N',N'')$. As $F_{\bw,\bd+\be} = F_{\bw,\bd} \oplus F_{\bw,\be}$, we may decompose $f = (f',f'')$ accordingly. Since $N$ is generated by $\im f$ as a $\C Q$-module, the quotient $N''$ is generated by $\im f''$. So $(N'',f'') \in R_\be^{\f,\st}$. Note that the submodule $N'$ need not be generated by $\im f'$. We hence obtain maps
$$
	R_\bd \times R_{\be}^{\f,\st} \xleftarrow[]{p} R_{\bd,\be}^{\f,\st} \xrightarrow[]{i} R_{\bd+\be}^{\f,\st}.
$$
We again pass to equivariant cohomology.
As the action of $G_\bd$ on $R_\bd^{\f,\st}$ is free and 
$\Hilb_{\bd,\bw} = R_\bd^{\f,\st}/G_{\bd}$ is a geometric $G_{\bd}$-quotient, we obtain the isomorphism
\begin{equation}
H_{G_\bd}^*(R_\bd^{\f,\st}) \iso H^*(\Hilb_{\bd,\bw}).
\end{equation} 
Under this identification, we get a map
\begin{equation}
*: H_{G_\bd}^*(R_\bd) \otimes H^*(\Hilb_{\be,\bw}) \to H^*(\Hilb_{\bd+\be,\bw})[-\chi(\bd,\be)]
\end{equation}
similarly to the definition of the product in the \coha.
We now define the \Ga-graded vector space
\begin{equation}
\cM_{\f}=\bop_{\bd\in\Ga}\cM_{\f,\bd},\qquad
\cM_{\f,\bd}=H^*(\Hilb_{\bd,\bw}).
\end{equation}
The map $*$ induces a structure of a $\Gamma$-graded $\mathcal{H}$-module on $\cM_\f$.
\begin{remark}
If the quiver $Q$ is symmetric, we define
$$
	\mathcal{M}_{\f,\bd} = H^*(\Hilb_{\bd,\bw})[-\chi(\bd,\bd)]
$$
and obtain a $\Gamma$-graded $\mathcal{H}$-module $\cM_\f$ in the category of $\bZ$-graded vector spaces.
\end{remark}

Consider the composition 
\begin{equation}
R_{\bd}^{\f,\st} \hookrightarrow R_{\bd}^\f = R_{\bd} \times F_{\bw,\bd} \to R_\bd
\end{equation}
of the open inclusion and the projection which forgets the framing vector. 
Pulling back gives rise to the map
$\phi: \mathcal{H}_\bd \to \mathcal{M}_{\f,\bd}$.

\begin{theorem}[{\cite[Thm.\ 3.2]{franzen_semia}}]
The map
\begin{equation}
\phi: \cH=\bop_\bd\cH_\bd \to \cM_\f=\bop_\bd\cM_{\f,\bd}
\end{equation}
is an epimorphism of $\Gamma$-graded $\mathcal{H}$-modules
and its kernel is given by
\begin{equation}
\ker\vi=\sum_{\bd'\ne0,\bd}\cH_\bd*(e_{\bd'}^\bw\cup\cH_{\bd'}),
\qquad
e_\bd^\bw 
=\prod_{i\in I}(x_{i,1}\dots x_{i,\bd_i})^{\bw_i},
\end{equation}
where the cup product $\cup$ corresponds to the usual product in $\cH_\bd
=\bigotimes_{i\in I}\bQ[x_{i,1},\dots,x_{i,\bd_i}]^{\Si_{\bd_i}}$.
\end{theorem}

\subsection{Tautological vector bundles}
We can describe the epimorphism $\vi:\cH\to\cM$ explicitly using tautological vector bundles.
For any vertex $i\in I$, we define the $G_\bd$-linearized
vector bundle $\cV_{\bd,i}$ of rank $\bd_i$ over $R_\bd=R(Q,\bd)$
to be the trivial vector bundle with the fiber over $M\in R_\bd$ equal to $M_i=\bC^{\bd_i}$,
equipped with the $G_\bd$-linearization induced by the projection $G_\bd\to\GL_{\bd_i}(\bC)$.
Under the isomorphism ~\eqref{coha as poly}
\begin{equation}
\mathcal{H}_\bd = H_{G_\bd}^*(R_\bd)\iso
\bts_{i \in I} \Q[x_{i,1},\ldots,x_{i,\bd_i}]^{\Sigma_{\bd_i}}
\end{equation}
the $G_\bd$-equivariant Chern class $c_k(\cV_{\bd,i})\in H^*_{G_\bd}(R_\bd)$ corresponds to the $k$\textsuperscript{th} elementary symmetric function 
\begin{equation}
c_k(\cV_{\bd,i})=e_k(x_{i,1},\ldots,x_{i,\bd_i}).
\end{equation}

Similarly, for any vertex $i\in I$, we define the \idef{$i$-th tautological vector bundle} $\cU_i=\cU_{\bd,i}$ 
to be the rank $\bd_i$ vector bundle over $\Hilb_{\bd,\bw}=R^{\f,\st}_\bd/G_\bd$, with the fiber over $M\in\Hilb_{\bd,\bw}$ equal to $M_i$.
More precisely, consider the trivial vector bundle of rank $\bd_i$ over $R^{\f,\st}_\bd$ with the  $G_\bd$-linearization induced by the projection $G_\bd\to\GL_{\bd_i}(\bC)$.
This bundle descends to the vector bundle $\cU_{\bd,i}$
over the quotient $\Hilb_{\bd,\bw}$.

The epimorphism $\vi$ considered earlier is given by
\begin{gather}
\vi:\cH_\bd\iso\bts_{i\in I}\bQ[x_{i,1},\dots,x_{\bd_i}]^{\Si_{\bd_i}}
\to\cM_{\f,\bd}=H^*(\Hilb_{\bd,\bw}),\\
e_k(x_{i,1},\ldots,x_{i,\bd_i})
=c_k(\cV_{\bd,i})\mto c_k(\cU_{\bd,i}),\qquad i\in I,\, k\ge1.
\end{gather}
Indeed, the vector bundle $\cU_{\bd,i}$ over $\Hilb_{\bd,\bw}$ arises as the descent of the trivial rank $\bd_i$ vector bundle over $R_\bd^{\f,\st}$,
with the $G_\bd$-linearization given by the $i$\textsuperscript{th} factor of $G_\bd$. 
This is precisely the pull-back of $\cV_{\bd,i}$ along the forgetful map $R_\bd^{\f,\st} \to R_\bd$.
\medskip

Because of the surjectivity of $\vi$, we conclude that $H^*(\Hilb_{\bd,\bw})$ is spanned by the vectors of the form
\begin{equation}
\prod_{i\in I}\prod_{k\ge1}c_k(\cU_{\bd,i})^{n_{i,k}}
\end{equation}
with $n_{i,k}\ge0$.
Later we will find a basis of $H^*(\Hilb_{\bd,\bw})$ consisting of vectors of this form.

\section{Degeneracy loci and tautological classes}

\subsection{Products of degeneracy classes}
\label{sec:deg loc}
Let $X$ be a smooth (irreducible) complex algebraic variety of dimension $N$.
Let $E_1,\ldots,E_n$ be vector bundles over $X$ of ranks $r_1,\ldots,r_n$.
For $1\le i\le n$, let $0\le k_i\le r_i$ and let
$$
	\sigma_i : \cO_X^{k_i+1} \to E_i
$$
be a morphism of vector bundles.
Consider the degeneracy locus
\begin{equation}\label{deg loc3}
D(\sigma_i) =D_{k_i}(\si_i)
=\sets{x \in X}{\rk\si_i(x) \leq k_i}
\end{equation}
with the structure of a closed subscheme of $X$ which comes from viewing it as the zero locus $Z\rbr{\bigwedge^{k_i+1} \sigma_i}$.
All irreducible components of $D(\si_i)$ have dimension at least $N-r_i+k_i$  (see \ful[Thm.~14.4]).
In \ful[\S 14.4] one defines a degeneracy class
\begin{equation}
\bD(\sigma_i)=\bD_{k_i}(\sigma_i) \in A_{N-r_i+k_i}(D(\sigma_i)).
\end{equation}
such that its image under the the push-forward along the closed embedding $D(\sigma_i) \emb X$ 
is equal to $c_{r_i-k_i}(E_i) \cap [X]$ (see \ful[Ex.~14.4.2]). Consider the scheme-theoretic intersection 
\begin{equation}
D = D(\sigma_1) \cap \ldots \cap D(\sigma_n).
\end{equation}
All of its irreducible components have dimension at least 
\begin{equation}
d = N - \sum_i(r_i-k_i).
\end{equation}
Our goal is to get a better understanding of the refined product \ful[\S8.1]
\begin{equation}\label{ref prod}
\bD=\bD(\sigma_1) \dots \bD(\sigma_n)\in A_d(D).
\end{equation}

\subsection{Refined products}
Given a morphism $f:X\to Y$ with smooth $Y$, subvarieties $X'\sbs X$, $Y'\sbs Y$ and classes $x\in A_*(X')$, $y\in A_*(Y')$, we consider the Cartesian diagram
\begin{ctikzcd}
X'\xx_Y Y'\rar\dar[hook]& X'\xx Y'\dar[hook]\\
X\rar{\ga_f}&X\xx Y
\end{ctikzcd}
where $\ga_f$ is the graph morphism of $f$
(note that $\ga_f$ is a regular embedding),
and we define the refined product \ful[\S8.1]
$$x\cdot_f y=\ga_f^!(x\xx y)\in A_*(X'\xx_YY'),$$
where $\ga_f^!$ is the refined Gysin morphism corresponding to the above diagram.
If $X=Y$ and $f=\id$, then we denote the above product by $x\cdot y\in A_*(X'\cap Y')$.

\begin{lemma}\label{pull-back}
Let $f:X\to Y$ be a morphism of smooth algebraic varieties (note that $f$ is automatically an \lci morphism), $Y_i\sbs Y$ be a subvariety and $y_i\in A_*(Y_i)$, for $i=1,2$.
Then we have
$$f^!(y_1\cdot y_2)=f^!(y_1)\cdot f^!(y_2)\in A_*(X\xx_Y(Y_1\cap Y_2)),$$
where $y_1\cdot y_2\in A_*(Y_1\cap Y_2)$ and $f^!(y_i)\in A_*(X\xx_YY_i)$, for $i=1,2$.
\end{lemma}
\begin{proof}
We obtain from associativity \ful[Prop.~8.1.1(a)] applied to $X\xto fY\xto\id Y$ that
$$[X]\cdot_f(y_1\cdot y_2)
=([X]\cdot_f y_1)\cdot_f y_2.$$
By \ful[Prop.~8.1.2] we have $f^!(y_i)=[X]\cdot_fy_i$ for $y_i\in A_*(Y_i)$.
Therefore
$$f^!(y_1\cdot y_2)=f^!(y_1)\cdot_f y_2$$
and we just need to show that $x\cdot_fy_2=x\cdot f^!(y_2)$ for $x=f^!(y_1)$.
But we have $x\cdot f^!(y_2)=x\cdot([X]\cdot_fy_2)
=(x\cdot[X])\cdot_f y_2=x\cdot_f y_2$.
\end{proof}

\subsection{Coefficient lemma}
Let $Y$ be a separated algebraic scheme over $\bC$.
Let $Z \sbs Y$ be an irreducible component of dimension $d$, equipped with the reduced subscheme structure, and let $Y'$ be the closure of $Y\ms Z$ (it is the union of all other irreducible components of $Y$),
also equipped with the reduced subscheme structure. 
There is an exact sequence
$$
	A_d(Z \cap Y') \to A_d(Z) \oplus A_d(Y') \to A_d(Y) \to 0
$$
(see \ful[Ex.\ 1.3.1]). As $\dim Z \cap Y' < \dim Z = d$, we obtain 
$$
A_d(Y) \iso A_d(Z) \oplus A_d(Y') 
= \Z[Z] \oplus A_d(Y').
$$
Therefore every cycle $y \in A_d(Y)$ can be written in a unique way as 
$$y = m\cdot [Z] + y',\qquad
m\in\bZ,\,y'\in A_d(Y').$$
We call $m$ the \emph{coefficient} of $[Z]$ in $y$.

\begin{lemma}\label{intersection lemma}
Using notation from \S\ref{sec:deg loc},
let	$U\sbs X$ be an open subset such that
\begin{enumerate}
\item $Z := D \cap U$ is irreducible and reduced of dimension $d$,
\item $D(\sigma_i) \cap U$ has pure dimension $N - r_i+k_i$, for all $1\le i\le n$.
\end{enumerate}
Then the closure $\bar{Z}\sbs D$ is an irreducible component of $D$ and, if we equip $\bar{Z}$ with the reduced subscheme structure, the coefficient of $[\bar{Z}]$ in the refined product
$\bD=\prod_i\bD(\si_i)\in A_d(D)$ is $1$.
\end{lemma}

\begin{proof}
We first show that $\bar{Z}$ is an irreducible component of $D$. 
It is irreducible being the closure of an irreducible set.
Therefore it is contained in some irreducible component $C$ of~$D$.
If $\bar{Z} \ne C$ then $\dim C > d$.
But $C$ intersects $U$ and thus $\dim C \cap U = \dim C > d = \dim Z$,
while $Z = D \cap U \sps C \cap U$.
A contradiction.

Now we will show that the coefficient of $[\bar{Z}]$ in $\bD$ is one.
We can write
$D = \bar{Z} \cup D'$, where $D'$ is the union of all other irreducible components of $D$.
Then
$$
\bD = m\cdot[\bar{Z}] + x,\qquad
m\in\bZ,\, x\in A_{d}(D').
$$
Consider the open embedding
$j: U \emb X$.
Then
$$
j^*\bD=m\cdot j^*[\bar Z]+j^*x=m\cdot[Z]+j^*x.
$$
(Note that $j^*[\bar{Z}] = [Z]$ as both $Z$ and $\bar{Z}$ are reduced, so $\bar{Z} \cap U = Z$ also scheme-theoretically.)
But $D \cap U = Z\sbs\bar Z$, hence 
$D' \cap U \sbs D'\cap \bar Z$.
Therefore $\dim D' \cap U\le \dim D'\cap\bar Z<d$
and thus $A_{d}(D' \cap U) = 0$ and $j^*x=0$.
We conclude that
$$j^*\bD=m\cdot[Z]$$
and we need to show that the coefficient of $[Z]$ in $j^*\bD$ is one.
The degeneracy locus
of the homomorphism 
$j^*\sigma_i: \OO_U^{k_i+1} \to E_i|_U$
is $D(j^*\sigma_i) = D(\sigma_i) \cap U$. 
As degeneracy classes are compatible with flat pull-backs \ful[Thm.\ 14.4(d)],
we obtain
$$
j^*\bD(\sigma_i)=\bD(j^*\sigma_i)
\in A_{N-r_i+k_i}(D(\sigma_i) \cap U),
$$
where, by abuse of notation, we use the letter $j$ also for the embedding $D(\sigma_i) \cap U \emb D(\sigma_i)$.
As refined products are compatible with pull-backs by Lemma \ref{pull-back}, we obtain
$$
j^*\bD 
=\prod_{i}j^*\bD(\sigma_i) 
=\prod_{i}\bD(j^*\sigma_i) 
\in A_{d}(D \cap U) = A_{d}(Z)=\bZ[Z].
$$
By assumption (2), the scheme
$$V_i=D(\si_i)\cap U=D(j^*\sigma_i),\qquad 1\le i\le n,$$
has pure dimension $N-r_i+k_i$.
Therefore by \ful[Thm.\ 14.4(c)], we get $$\bD(j^*\sigma_i) 
=[D(j^*\sigma_i)]=[V_i],
\qquad 1\le i\le n.$$
This implies that (see \ful[Ex.~8.2.1])
$$
j^*\bD=\prod_i\bD(j^*\si_i)=\prod_i[V_i]
=i(Z,V_1\cdot\ldots\cdot V_n,U)\cdot[Z]\in A_d(Z).
$$
We know by \ful[Prop.\ 8.2] that 
$m=i(Z,V_1\cdot\ldots\cdot V_n,U)$
satisfies
$1 \leq m \leq l(\cO_{D\cap U,Z})$, where $\cO_{D \cap U,Z}$ is the local ring of $D\cap U=\cap_i V_i$ at the generic point of the irreducible component $Z$. 
But as $D \cap U = Z$ (which is reduced and irreducible), the length of this artinian local ring is one, hence $m=1$.
\end{proof}

\subsection{Tautological vector bundles and their sections}
As before, let $Q$ be a quiver with the set of vertices $I$.
Let $\bw\in \bN^I$ and $Q^\bw$ be the corresponding framed quiver.
Let $\cP=\cP^\f$ be the set of paths in $Q^\bw$ that start at $\infty$.
Given a dimension vector $\bd$, we define $\cN_\bd=\cN_{\bd,\f}$ to be the non-commutative Hilbert scheme parameterizing stable framed representations of dimension ~$\bd$.
Recall that its dimension is equal to
\begin{equation}
N=\dim\cN_\bd=\bw\cdot\bd-\hi(\bd,\bd).
\end{equation}

For any vertex $i\in I$, we have
the $i$-th tautological vector bundle 
$\cU_i=\cU_{\bd,i}$ over $\cN_\bd$, with the fiber over $M\in\cN_\bd$ equal to $M_i$.
This vector bundle has rank $\bd_i$.
For any path $u\in\cP_i$, we have a vector $m_u\in M_i$.
The family of such vectors over all $M\in\cN_\bd$ induces a section
\begin{equation}
m_u\in\Ga(\cN_\bd,\cU_i).
\end{equation}
Given a tree $S\in\cP(\bd)$ and $v\in C(S)_i$,
we define \eqref{Sv}
\begin{equation}
S(v)=\sets{u\in S_i}{u<v}\cup\set v,\qquad k_v=\#S(v)-1.
\end{equation}
Then the above sections induce a morphism of vector bundles
\begin{equation}
\si_v=(m_u)_{u\in S(v)}:
\cO^{S(v)}=\cO^{k_v+1}\to \cU_i.
\end{equation}
Note that $k_v\le \bd_i=\rk\cU_i$.
The degeneracy locus $D(S(v))\sbs\cN_\bd$ 
considered in \eqref{deg loc1}
is equal to the degeneracy locus $D(\si_v)$
considered in \S\ref{sec:deg loc}.
By the discussion in \S\ref{sec:deg loc},
we have a degeneracy class
$$\bD(\si_v)\in A_{N-\bd_i+k_v}(D(\si_v))$$
such that its image in $A_{N-\bd_i+k_v}(\cN_\bd)$ is equal to
$$c_{\bd_i-k_v}(\cU_i)\cap[\cN_\bd].$$

Let us consider the degeneracy locus of the tree $S$ \eqref{tree deg locus}
\begin{equation}
D_S
=\bigcap_{v\in C(S)}D(S(v))
=\bigcap_{v\in C(S)}D(\si_v)\sbs\cN_\bd
\end{equation}
and the class
\begin{equation}
\bD_S=\prod_{v\in C(S)}\bD(\si_v)
\in A_m(D_S),\qquad m=N-\sum_{v\in C(S)}(\bd_{t(v)}-k_v).
\end{equation}
We obtain from \eqref{dim Z_S}
and \eqref{dim U_S} that
\begin{equation}
d(S)=\sum_{v\in C(S)}k_v,\qquad
N=\udim C(S)\cdot\bd=\sum_{v\in C(S)}\bd_{t(v)}.
\end{equation}
Therefore
$$m=\sum_{v\in C(S)}k_v=d(S)=\dim Z_S.$$
The image of the class $\bD_S\in A_{d(S)}(D_S)$ in $A_{d(S)}(\cN_\bd)$ is equal to
$$\bD'_S=\rbr{\prod_{i\in I}\prod_{v\in C(S)_i}c_{\bd_i-k_v}(\cU_i)}\cap[\cN_\bd].$$

We know that $A_{d(S)}(\cN_\bd)$ has a basis consisting of $[\bar Z_{S'}]$ with $d(S')=d(S)$.
By Lemma \ref{D_S inclusion}, we have $D_S\sbs\bigsqcup_{S'\ge S}Z_{S'}$.
Therefore we have an expression in $A_{d(S)}(\cN_\bd)$
\begin{equation}\label{D-Z}
\bD'_S
=\sum_{\ov{S'\ge S}{d(S')=d(S)}}n_{S',S}[\bar Z_{S'}]
\end{equation}
with $n_{S',S}\in\bZ$.
Applying Lemma \ref{intersection lemma} to the degeneracy loci $D(\si_v)$ and the open subset $U_S\sbs\cN_\bd$ (satisfying $D_S\cap U_S=Z_S$ by Lemma \ref{D_S inclusion}), we obtain $n_{S,S}=1$.
This implies

\begin{theorem}\label{basis1}
The classes
$$\prod_{i\in I}\prod_{v\in C(S)_i}c_{\bd_i-k_v}(\cU_i)\in A^{N-d(S)}(\cN_\bd),\qquad S\in\cP(\bd),$$
form a basis of $A^*(\cN_\bd)$.
\end{theorem}

Note that the numbers $k_v$ (hence the above description of a basis of $A^*(\cN_\bd)$) depend on the admissible total order on $\cP$.

\begin{theorem}\label{th:basis2}
The classes
$$\prod_{i\in I}\prod_{k\ge1} c_k(\cU_i)^{\tp\la i_k-\tp\la i_{k+1}}\in A^{\n\la}(\cN_\bd),\qquad
\la\in\cS(\bd),$$
form a basis of $A^*(\cN_\bd)$, where $\n\la=\sum_{i\in I}\n{\tp\la i}$.
\end{theorem}
\begin{proof}
Recall that for any admissible order on \cP we have a bijection between the set of trees $\cP(\bd)$ and the set of multipartitions $\cS(\bd)$ (see Theorem \ref{bijection}).
If $\la\in\cS(\bd)$ corresponds to the tree $S\in\cP(\bd)$ under this bijection, then
$$\tp\la i_k-\tp\la i_{k+1}=\#\sets{v\in C(S)_i}{\bd_i-k_v=k}.$$
Therefore
$$\prod_i\prod_{v\in C(S)_i}c_{\bd_i-k_v}(\cU_{i})
=\prod_{i\in I}\prod_{k\ge1} c_k(\cU_i)^{\tp\la i_k-\tp\la i_{k+1}}$$
and we apply Theorem \ref{basis1}. 
\end{proof}

Note that by Theorem \ref{th:cell-dec-BM}, we have $\HB_i(\cN_\bd,\bZ)=0$ for odd $i$ and $\HB_{2i}(\cN_\bd,\bZ)\iso A_i(\cN_\bd)$.
Similarly, $H^i(\cN_\bd,\bQ)=0$ for odd $i$ and $H^{2i}(\cN_\bd,\bQ)\iso A^i(\cN_\bd)\ts\bQ$.
Therefore the above theorem gives a canonical basis of $H^*(\cN_\bd,\bQ)$.

\subsection{Dependence of classes on the admissible order}
\label{sec:dependence}
Given an admissible order on \cP, we have a bijection $\cS(\bd)\to\cP(\bd)$, $\la\mto S_\la$.
In the previous section we have seen that the classes $\bD'_\la=\bD'_{S_\la}$ forming a basis of $A_*(\cN_\bd)$ are independent of the choice of an admissible order on \cP.
One may ask if the classes $[\bar Z_{S_\la}]$ (forming another basis of $A_*(\cN_\bd)$) are independent of the choice of an admissible order.
We will see in this section that this is not always the case.

Consider the quiver $Q^\bw$
$$\begin{tikzcd}
\infty\ar[r,"f"]& 0
\ar[loop, out=10,in=45,looseness=5,"a"']
\ar[loop, out=-10,in=-45,looseness=5,"b"]
\end{tikzcd}
$$
with arrows ordered as $f<a<b$.
Let $\cP$ be the set of paths in $Q^\f$ starting at $\infty$.
Assume that \cP is equipped with the shortlex order.
Then the trees in $\cP(3)$ written in the increasing order
(see \S\ref{sec:cell dec})
 are (we omit the root $*$)
\begin{table}[ht]
\begin{tabular}{|c|c|c|c|}
\hline
Tree $S$ & $d(S)$ & Partition\\
\hline
$S_1=(f,af,bf)$&12&()\\
$S_2=(f,af,a^2f)$&11&(1)\\
$S_3=(f,af,baf)$&10&(2)\\
$S_4=(f,bf,abf)$&10&(1,1)\\
$S_5=(f,bf,b^2f)$&9&(2,1)\\
\hline
\end{tabular}
\end{table}

For a tree $S\in\cP(3)$ with the corresponding partition $\la$, let
$$
\bD'_S=\bD'_\la=\rbr{\prod_{k\ge1} c_k(\cU_0)^{\la_k-\la_{k+1}}}\cap[\cN_3]\in
A_{d(S)}(\cN_3),$$
where $d(S)=\dim\cN_3-\n\la=12-\n\la$.
The group $A_{10}(\cN_3)$ has a basis consisting of 
$[\bar Z_{S_3}]$ and $[\bar Z_{S_4}]$ and a basis consisting of $\bD'_{S_3}=c_1(\cU_0)^2\cap[\cN_3]$ and 
$\bD'_{S_4}=c_2(\cU_0)\cap[\cN_3]$.
By \eqref{D-Z} we have
$$\bD'_{S_3}=[\bar Z_{S_3}]+n[\bar Z_{S_4}],\qquad \bD'_{S_4}=[\bar Z_{S_4}],$$
for some $n\in\bZ$.
In order to find the multiplicity $n$ of $[\bar Z_{S_4}]$ in $\bD'_{S_3}$, we will study the intersection $U_{S_4}\cap D_{S_3}$.
We have
$$C(S_3)=(bf,a^2f,abaf,b^2af),\qquad
C(S_4)=(af,b^2f,a^2bf,babf).$$
Let us denote the elements of $S_4$ by $u_1<u_2<u_3$ and the elements of $C(S_4)$ by $v_1<v_2<v_3<v_4$.
Then the coordinates $(c_{ij})$ of $M\in U_{S_4}$ are given by $m_{v_j}=\sum_i c_{ij}m_{u_i}$ 
(\cf Lemma~\ref{lm:dim U_S}).
For $M\in D_{S_3}$ we require that $(m_f,m_{af},m_{bf})$ are linearly dependent and $(m_f,m_{af},m_{a^2f})$ are linearly dependent.
The first condition means that $c_{31}=0$ and $m_{af}=c_{11}m_f+c_{21}m_{bf}$.
For the second condition, we note that
$$m_{a^2f}=am_{af}=c_{11}m_{af}+c_{21}m_{abf}
=c_{11}(c_{11}m_f+c_{21}m_{bf})+c_{21}m_{abf}.
$$
Therefore we require that
$$\det\pmat{1&c_{11}&c_{11}^2\\
0&c_{21}&c_{11}c_{21}\\
0&0&c_{21}}=c_{21}^2=0.$$
For $M\in Z_{S_4}$, we require that $m_{af}$ is contained in $\ang{m_f}$, hence $c_{21}=c_{31}=0$.
We conclude that the multiplicity of $[\bar Z_{S_4}]$ in $\bD'_{S_3}$ is $n=2$, hence 
$$[\bar Z_{S_3}]=\bD'_{S_3}-2\bD'_{S_4}
,\qquad
[\bar Z_{S_4}]=\bD'_{S_4}.
$$

On the other hand, let us equip \cP with the lex order.
Then the trees in $\cP(3)$ written in the increasing order
are (we omit the root $*$)
\begin{table}[ht!]
\begin{tabular}{|c|c|c|c|}
\hline
Tree $S$& $d(S)$ & Partition\\
\hline
$S_1=(f,af,a^2f)$&12&()\\
$S_2=(f,af,baf)$&11&(1)\\
$S_3=(f,af,bf)$&10&(2)\\
$S_4=(f,bf,abf)$&10&(1,1)\\
$S_5=(f,bf,b^2f)$&9&(2,1)\\
\hline
\end{tabular}
\end{table}

Note that the order of partitions is the same as before by Lemma~\ref{lm:part-order}.
Using the same method as before one can show that (see \cite[\S2.2]{franzen_cohomology})
$$[\bar Z_{S_3}]=\bD'_{S_3}-4\bD'_{S_4},\qquad 
[\bar Z_{S_4}]=\bD'_{S_4}.$$
This implies that while the classes $\bD'_{S_i}$ are independent of the choice of an admissible order, the classes $[\bar Z_{S_i}]$ may depend on it.

\bibliography{biblio}
\bibliographystyle{hamsplain}
\end{document}